\documentclass[11pt,a4paper]{article}
\usepackage[american]{babel}
\usepackage{verbatim}
\usepackage{a4}
\usepackage{hyperref}
\usepackage{amsmath, amssymb, amsthm}
\usepackage{subfig}
\usepackage{cases}
\usepackage{ifthen}
\usepackage{tikz}
\usetikzlibrary{positioning,arrows,patterns,decorations.markings, arrows.meta, calc,shapes.geometric}
\usepackage{tkz-euclide}
\usepackage{dsfont}
\usepackage{graphicx}
\usepackage{subcaption}
\usepackage[nice]{nicefrac}
\usepackage{mathtools}

\newtheorem{theorem}{Theorem}[section]
\newtheorem{proposition}[theorem]{Proposition}
\newtheorem{lemma}[theorem]{Lemma}
\newtheorem{corollary}[theorem]{Corollary}

\theoremstyle{definition}

\newtheorem*{definition}{Definition}
\newtheorem*{example}{Example}
\newtheorem*{remark}{Remark}

\newcommand{\R}{\mathds{R}}
\newcommand{\Z}{\mathds{Z}}

\newcommand{\Q}{\mathds{Q}}

\renewcommand{\Re}{\operatorname{Re}}

\DeclareMathOperator{\fix}{\mathrm{Fix}}

\begin{document}
\title{Real Riemann Surfaces: Smooth and Discrete}

\author{Johanna D\"untsch$^{1}$ \and Felix G\"unther$^{2}$}

\footnotetext[1]{Institut f\"ur Mathematik, Technische Universit\"at Berlin, Stra{\ss}e des 17. Juni 136, 10623 Berlin, Germany}
\footnotetext[2]{Institut f\"ur Differentialgeometrie, Leibniz Universit\"at Hannover, Welfengarten 1, 30167 Hannover, Germany. E-mail: felix.guenther@math.uni-hannover.de}

\maketitle
\setlength{\parindent}{0em}

\begin{abstract}
	\noindent
	This paper develops a discrete theory of real Riemann surfaces based on quadrilateral cellular decompositions (quad-graphs) and a linear discretization of the Cauchy-Riemann equations. We construct a discrete analogue of an antiholomorphic involution and classify the topological types of discrete real Riemann surfaces, recovering the classical results on the number of real ovals and the separation of the surface. 
	Central to our approach is the construction of a symplectic homology basis adapted to the discrete involution. Using this basis, we prove that the discrete period matrix admits the same canonical decomposition $\Pi = \frac{1}{2} H + i T$ as in the smooth setting, where $H$ encodes the topological type and $T$ is purely imaginary. This structural result bridges the gap between combinatorial models and the classical theory of real algebraic curves.
	\\ \vspace{0.5ex}
	
	\vspace{0.5ex}
	
\end{abstract}
		
\raggedbottom
\setlength{\parindent}{0pt}
\setlength{\parskip}{1ex}

\section{Introduction}\label{sec:intro}

Discrete complex analysis has evolved into a powerful field that bridges classical function theory, geometry, and statistical physics.
While early foundations were laid by Isaacs \cite{Is41}, Lelong-Ferrand \cite{Fe44}, and Duffin \cite{Du56}, the field has seen a renaissance in recent decades, driven by the study of integrability and conformal invariance in lattice models \cite{Sm10S, Ke02, ChSm11}.
A central object of study in this linear theory is the concept of a \emph{discrete Riemann surface}, modeled on quadrilateral cellular decompositions (quad-graphs) equipped with complex weights. This framework, developed significantly by Mercat \cite{Me01, Me07} and Bobenko and the second author \cite{BoG15, BoG17}, allows for the discretization of key structures such as holomorphic forms, period matrices, and the Abel-Jacobi map.

Despite these advances, a discrete theory of \emph{real Riemann surfaces} -- compact Riemann surfaces endowed with an antiholomorphic involution -- has been largely absent.
In the smooth setting, real Riemann surfaces are fundamental objects. They correspond to real algebraic curves \cite{Sil01, Vin93} and play a crucial role in the construction of real solutions to non-linear integrable equations, such as the Kadomtsev-Petviashvili equation \cite{BBEIM94}.
Classically, the presence of an antiholomorphic involution $\tau$ imposes strong topological constraints on the surface, most notably Harnack's inequality regarding the number of fixed components (real ovals), and forces the period matrix to take a specific canonical form.

From this classical point of view, real Riemann surfaces and their period matrices encode fundamental topological and analytic information, dating back to the work of Weichold and later developments in real algebraic geometry.
The present paper shows that this structure admits a faithful discrete analogue: discrete real Riemann surfaces carry period matrices with the same canonical decomposition as in the smooth case, determined purely by topological data and discrete harmonic theory.
This constitutes the first systematic discrete analogue of the theory of real Riemann surfaces, including a discrete version of the canonical period matrix decomposition.

\subsection{Contribution of this paper}

In this paper, we establish a complete discrete theory of real Riemann surfaces. We work in the setting of bipartite quad-graphs equipped with a discrete complex structure (complex weights). Our main contributions are:
\begin{enumerate}
	\item \textbf{Topological Classification:} We define a discrete antiholomorphic involution on quad-graphs and prove that the discrete surface obeys the same topological classification as the continuous one. We distinguish between dividing and non-dividing types and verify the bounds on the number of discrete real ovals (Theorem~\ref{th:number_ovals_discrete}).
	\item \textbf{Adapted Homology Basis:} We construct a symplectic basis of the discrete homology that is adapted to the involution. Our approach relies on a continuous realization of the discrete surface, enabling us to transfer the topological existence results and properties of the canonical homology basis from the smooth theory directly to the combinatorial setting.
	\item \textbf{Discrete Period Matrix:} We prove that the period matrix $\Pi$ of a discrete real Riemann surface satisfies the identity
	\[
	\Pi = \frac{1}{2} H + i T,
	\]
	where $H$ is a matrix with integer entries reflecting the topological type, and $T$ is a real matrix. This mirrors the classical result exactly (Theorem~\ref{th:completeperiodmatrix}). In particular, for discrete M-curves, the period matrix is purely imaginary.
\end{enumerate}

\subsection{Discrete versus continuous theory}

A recurring theme in discrete differential geometry is the challenge of defining discrete objects that preserve the essential symmetries and structures of the smooth theory.
The convergence of discrete period matrices to their continuous counterparts has been established for triangulations \cite{BoSk12} and recently for general quadrangulations \cite{G23}.
However, the structural properties of these matrices \emph{before} the limit is taken are of independent interest, both for numerical stability and for understanding the combinatorial nature of the theory.
To ensure that our discrete construction is natural, we first revisit the continuous theory in Section~\ref{sec:continuous}. We provide an elementary derivation of the classical period matrix structure, specifically tailored to be transferable to the discrete graph setting. This elementary treatment, which avoids heavy algebraic geometry machinery, serves as the blueprint for our discrete definitions in Section~\ref{sec:discrete_real}.

\subsection{Outline}

The paper is structured as follows.
Section~\ref{sec:continuous} reviews the topology of real Riemann surfaces and derives the canonical homology basis in the smooth setting.
We present these classical results in a constructive manner that anticipates the discretization.
In Section~\ref{sec:Riemann}, we provide the necessary background on discrete Riemann surfaces, including the definition of the medial graph, discrete forms, and the discrete complex structure.
Section~\ref{sec:discrete_real} contains the main results: we define discrete real Riemann surfaces, analyze their topology, and derive the characterization of the discrete period matrix.
Finally, we discuss explicit constructions of discrete real Riemann surfaces of various genera and types.
Algebraic details on bilinear forms over $\mathds{Z}_2$, which are essential for the proofs, are collected in Appendix~\ref{sec:bilinearforms}.


\section{Real Riemann surfaces}\label{sec:continuous}

The primary objective of this section is to derive a canonical form for the period matrix of a real Riemann surface in a manner that allows for direct discretization. While the results presented here are classical and appear in various works \cite{Vin93, Wei83, BBEIM94}, standard derivations often rely on advanced algebraic geometry machinery (see the discussion at the end of Section~\ref{sec:homology}). Such techniques are not immediately transferable to the combinatorial setting of discrete Riemann surfaces.

To facilitate the discretization in Section~\ref{sec:discrete_real}, we restrict our exposition to arguments that are either purely topological or linear. Topological properties, such as the separation of the surface by fixed-point sets, are naturally preserved in the graph-theoretical framework. Linear relations, such as the properties of the period matrix, can be adapted using the linear discretization theory for holomorphic forms presented in Section~\ref{sec:Riemann}.
Consequently, we provide a self-contained and elementary introduction to the theory of real Riemann surfaces, synthesizing classical arguments with a perspective tailored to our discrete goals.
Concepts defined here on smooth manifolds -- such as the separation of the surface by fixed-point sets -- will correspond directly to cuts and cycles on the quad-graph in the discrete setting.

Throughout this section, let $\Sigma$ be a compact Riemann surface of genus $g$. A \emph{real Riemann surface} is defined as a pair $(\Sigma, \tau)$, where $\Sigma$ is a compact Riemann surface and $\tau: \Sigma \to \Sigma$ is an antiholomorphic involution (a conformal map satisfying $\tau \circ \tau = \mathrm{id}$ which is antiholomorphic in local coordinates).

\subsection{Real ovals and topological classification}\label{sec:realovals}

We first characterize the fixed point set of the involution. The following local result is standard; see, for example, \cite{J04}.

\begin{lemma}\label{lem:fixedpoint}
	Let $\mathds{D} \subset \mathds{C}$ be the open unit disk centered at the origin, and let $f : \mathds{D} \rightarrow \mathds{D}$ be an antiholomorphic involution with $f(0) = 0$. Then, the set of fixed points of $f$ forms a straight line segment passing through the origin.
\end{lemma}

\begin{proof}
	Since $f$ is antiholomorphic, the function $\bar{f}$ (defined by $z \mapsto \overline{f(z)}$) is holomorphic with $\bar{f}(0) = 0$. By the Schwarz Lemma, $|\bar{f}(z)| \leq |z|$ for all $z \in \mathds{D}$. Applying the involution property $f(f(z)) = z$, we obtain
	\[
	|z| = |\bar{f}(\bar{f}(z))| \leq |\bar{f}(z)| \leq |z|.
	\]
	Equality in the Schwarz Lemma implies that $\bar{f}(z) = e^{i\theta} z$ for some $\theta \in \mathds{R}$, representing a rotation. Consequently, $f(z) = e^{-i\theta} \bar{z}$. This is a reflection across the line $\{ r e^{-i\theta/2} \mid r \in \mathds{R} \}$. The intersection of this line with $\mathds{D}$ is the set of fixed points.
\end{proof}

Using this local linearization, we obtain the global structure of the fixed point set.

\begin{corollary}\label{cor:fixedpoint}
	Let $\tau: \Sigma \to \Sigma$ be an antiholomorphic involution. The fixed point set $\fix(\tau) = \{p \in \Sigma \mid \tau(p)=p\}$ consists of a disjoint union of simple closed curves.
\end{corollary}

\begin{proof}
	Suppose $\fix(\tau)$ is non-empty. Let $p \in \fix(\tau)$. Since $\tau$ is continuous, $\fix(\tau)$ is closed. 
	Let $V$ be a simply connected coordinate neighborhood of $p$. Define $U \coloneq V \cap \tau(V)$. Then $p \in U$ and $\tau(U) = U$.
	By the Riemann Mapping Theorem, there exists a biholomorphic map $z: U \to \mathds{D}$ mapping the open symmetric neighborhood $U$ to the unit disk $\mathds{D}$ such that $z(p) = 0$.
	
	The induced map $\tilde{\tau} = z \circ \tau \circ z^{-1}$ is an antiholomorphic involution of $\mathds{D}$ fixing the origin. By Lemma~\ref{lem:fixedpoint}, the fixed points of $\tilde{\tau}$ in $\mathds{D}$ form a straight line segment. Consequently, the fixed points of $\tau$ in $U$ form a curve segment passing through $p$. Since $\Sigma$ is compact, these local segments connect to form a finite number of disjoint closed curves.
\end{proof}

\begin{definition}
	The connected components of $\fix(\tau)$ are called \emph{real ovals}. We denote their number by $k$.
\end{definition}

\begin{remark}[Discrete preview]
	In Section~\ref{sec:discrete_real}, the smooth surface $\Sigma$ will be replaced by a bipartite quad-graph $\Lambda$, and the involution $\tau$ will become a graph automorphism reversing the orientation of faces. The real ovals will correspond to cycles of edges and diagonals fixed by $\tau$ (in the color-preserving case) or cycles of segments connecting the midpoints of opposite edges of the quadrilaterals (in the color-reversing case).
\end{remark}

The topology of a real Riemann surface is determined by how these ovals separate the surface. The following lemma is due to Weichold \cite{Wei83}. We provide an alternative proof utilizing a coloring argument that is particularly intuitive.

\begin{lemma}\label{lem:12components}
	Let $(\Sigma, \tau)$ be a real Riemann surface. The complement $\Sigma \setminus \fix(\tau)$ consists of either one or two connected components.
	\begin{itemize}
		\item If it consists of one component, the quotient $\Sigma/\tau$ is non-orientable.
		\item If it consists of two components, they are interchanged by $\tau$, and $\Sigma/\tau$ is orientable.
	\end{itemize}
\end{lemma}

\begin{proof}
	Since the real ovals are disjoint simple closed curves, they form the boundary of the connected components of $\Sigma \setminus \fix(\tau)$. We assign a color to each component and show that two colors suffice.
	
	Let $x_0$ be a real oval. Locally, $x_0$ separates a neighborhood into two sides. Assume the component on one side is colored $A$ and the component on the other side is colored $B$. Let $x_1$ be another oval (possibly $x_0$ itself) forming part of the boundary of the $A$-colored component. Consider a path $\gamma$ within the $A$-component connecting a point near $x_0$ to a point near $x_1$.
	Since $\gamma$ avoids $\fix(\tau)$, the image path $\tau(\gamma)$ lies entirely in $\Sigma \setminus \fix(\tau)$. Because $\tau$ fixes the ovals but swaps the local sides (as it is an orientation-reversing reflection locally), $\tau(\gamma)$ must connect the $B$-side of $x_0$ to the $B$-side of $x_1$.
	
	Since $\Sigma$ is connected, any two points in $\Sigma \setminus \fix(\tau)$ can be connected by a path. This path determines a chain of adjacent components. The argument above shows that traversing an oval always flips the side (from $A$ to $B$ or vice versa). Thus, we can consistently color the entire complement with at most two colors.
	\begin{itemize}
		\item If $A \neq B$, $\Sigma \setminus \fix(\tau)$ has two components $R_1, R_2$ with $\tau(R_1) = R_2$. The quotient $\Sigma/\tau$ is homeomorphic to the closure of $R_1$, which is an orientable surface with boundary.
		\item If $A = B$, there is a path connecting the two local sides of an oval without crossing $\fix(\tau)$. $\Sigma \setminus \fix(\tau)$ is connected. Since $\tau$ reverses orientation, the projection to the quotient yields a non-orientable surface. \qedhere
	\end{itemize}
\end{proof}

Based on this lemma, we classify real Riemann surfaces as follows:

\begin{definition}
	A real Riemann surface $(\Sigma, \tau)$ is called:
	\begin{enumerate}
		\item[(i)] \textbf{Dividing}, if $\Sigma \setminus \fix(\tau)$ is disconnected (two components).
		\item[(ii)] \textbf{Non-dividing}, if $\Sigma \setminus \fix(\tau)$ is connected.
	\end{enumerate}
	If the number of real ovals attains the maximum possible value $k = g+1$ (see Proposition~\ref{prop:number_ovals}), the surface is called an \emph{M-curve}.
\end{definition}

\begin{example}
	Figure~\ref{fig:dividing6} illustrates a dividing surface of genus 6. The involution represents a reflection across the horizontal plane, exchanging the upper and lower halves.
	Figure~\ref{fig:non-dividing12} depicts non-dividing surfaces. On the left (genus 1), the involution maps the upper triangle to the lower one (with appropriate identification), yielding a single connected complement. In the right figure (genus 2), $\tau$ maps the right quadrilateral to the upper one; for example, we have $\tau(a_1) = b_1$.
	Hyperelliptic curves defined by $y^2 = P(x)$ with a real polynomial $P$ having all real roots are classic examples of M-curves.
\end{example}

\begin{figure}[!ht]
	\centering
	\includegraphics[width=0.7\textwidth]{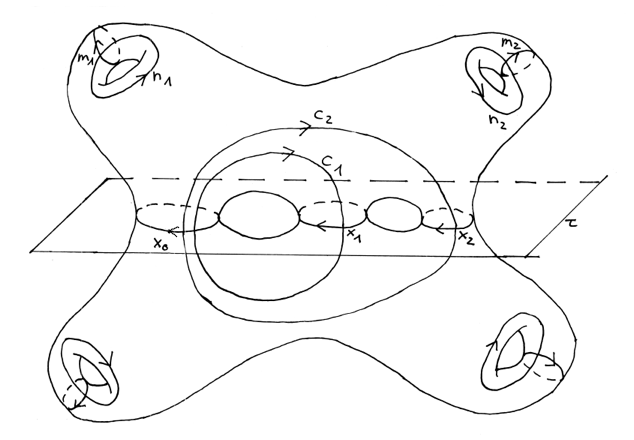}
	\caption{A dividing real Riemann surface of genus $g=6$ with $k=3$ real ovals ($x_0, x_1, x_2$). The involution $\tau$ reflects across the plane containing the ovals, swapping the two components. The cycles $m_i, n_i, c_i$ indicate the symplectic basis elements derived in Proposition~\ref{prop:realhomology}.}
	\label{fig:dividing6}
\end{figure}

\begin{figure}[htbp]
	\centering
	\subfloat[Genus 1 (Torus)]{
		\begin{tikzpicture}[scale=3.5, >=Latex, thick]
			\coordinate (A) at (0,0);
			\coordinate (B) at (0,1);
			\coordinate (C) at (1,1);
			\coordinate (D) at (1,0);
			
			\draw[dashed] (B) -- (D) node[midway, above right] {$\fix(\tau)$};
			
			\draw[->] (A) -- (B) node[midway, left] {$b$};
			\draw[-] (D) -- (C);
			
			\draw[->] (A) -- (D) node[midway, below] {$a$};
			\draw[-] (B) -- (C);
		\end{tikzpicture}
	}
	\qquad
	\subfloat[Genus 2]{
		\begin{tikzpicture}[scale=2.8, >=Latex, thick]
			\coordinate (O) at (0,0);
			\coordinate (A1) at (1.4,0);
			\coordinate (B1) at (0,1.4);
			\coordinate (A2) at (1.4,0.7);
			\coordinate (B2) at (0.7,1.4);
			\coordinate (C) at (0.7,0.7);
			
			\draw (O) -- (B1) -- (B2) -- (C) -- (A2) -- (A1) -- cycle;
			
			\draw[dashed] (O) -- (C) node[midway, below right] {$\fix(\tau)$};
			
			\draw[->] (O) -- (A1) node[midway, below] {$a_1$};
			\draw[->] (C) -- (A2) node[midway, below] {$a_2$};
			\draw[->] (O) -- (B1) node[midway, left] {$b_1$};
			\draw[->] (C) -- (B2) node[midway, left] {$b_2$};
		\end{tikzpicture}
	}
	\caption{Examples of non-dividing real Riemann surfaces using polygonal representations with identified edges. The dashed lines represent the fixed point set $\fix(\tau)$. Both examples exhibit $k=1$ oval but remain connected after removing $\fix(\tau)$.}
	\label{fig:non-dividing12}
\end{figure}
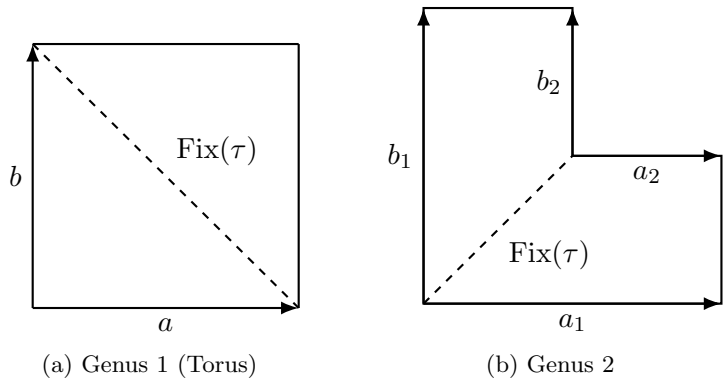

We conclude this subsection with the classical bound on the number of real ovals, known as Harnack's inequality. We present a topological proof based on the Euler characteristic, which highlights the structural relationship between $\Sigma$ and the quotient $\Sigma/\tau$.

\begin{proposition}\label{prop:number_ovals}
	Let $(\Sigma, \tau)$ be a real Riemann surface of genus $g$ with $k$ real ovals.
	\begin{enumerate}
		\item[(i)] \textbf{Harnack's Inequality:} $0 \leq k \leq g + 1$.
		\item[(ii)] If $k = g + 1$ (M-curve), then $\Sigma$ is dividing. If $k = 0$, then $\Sigma$ is non-dividing.
		\item[(iii)] If $\Sigma$ is dividing, then $k \equiv g + 1 \pmod{2}$.
	\end{enumerate}
\end{proposition}

\begin{proof}
	Let $\chi$ denote the Euler characteristic. Consider the quotient surface $\Sigma/\tau$. This is a surface with boundary consisting of $k$ boundary circles (the images of the real ovals). Let $g'$ be the genus of the closed surface obtained by capping these boundary circles with disks.
	
	The Euler characteristic of $\Sigma$ and $\Sigma/\tau$ satisfies the relation
	\begin{equation}\label{eq:Euler}
		\chi(\Sigma) = 2 \chi(\Sigma/\tau).
	\end{equation}
	This follows from considering a triangulation of $\Sigma/\tau$ where the boundary lies on edges. Lifting this triangulation to $\Sigma$ doubles the number of faces, edges, and vertices, except for those on the fixed point set (boundary), which are their own preimages. Since $\chi = V - E + F$, the contribution from the bulk doubles, while the boundary contribution cancels out in the relation relative to the double cover.
	
	Using $\chi(\Sigma) = 2 - 2g$, we analyze the two cases from Lemma~\ref{lem:12components}:
	
	\begin{enumerate}
		\item[(i)]
		\textbf{Case 1: $\Sigma$ is dividing.} Then $\Sigma/\tau$ is orientable. Its Euler characteristic is $\chi(\Sigma/\tau) = 2 - 2g' - k$. Substituting into \eqref{eq:Euler}:
		\[
		2 - 2g = 2(2 - 2g' - k) \quad \Rightarrow \quad 1 - g = 2 - 2g' - k \quad \Rightarrow \quad k = g + 1 - 2g'.
		\]
		Since $g' \geq 0$, we have $k \leq g + 1$.
		
		\textbf{Case 2: $\Sigma$ is non-dividing.} Then $\Sigma/\tau$ is non-orientable. Its Euler characteristic is $\chi(\Sigma/\tau) = 2 - g' - k$ (where $g'$ represents the number of cross-caps/genus of the non-orientable surface). Substituting into \eqref{eq:Euler}:
		\[
		2 - 2g = 2(2 - g' - k) \quad \Rightarrow \quad 1 - g = 2 - g' - k \quad \Rightarrow \quad k = g + 1 - g'.
		\]
		Since $g' \geq 1$ (a non-orientable surface must have at least one cross-cap), we strictly have $k \leq g$.
		
		Combining both cases, we obtain $k \leq g + 1$.
		
		\item[(ii)] If $k = g + 1$, the formula from Case 2 ($k = g + 1 - g'$) would imply $g' = 0$, which is impossible for a non-orientable surface. Thus, an M-curve must fall into Case 1 (dividing).
		If $k = 0$, the fixed point set is empty, so $\Sigma \setminus \fix(\tau) = \Sigma$ is connected (non-dividing).
		
		\item[(iii)] This follows directly from the derivation in Case 1 above: $k = (g+1) - 2g'$, hence $k$ and $g+1$ have the same parity.
	\end{enumerate}
\end{proof}

\subsection{Homology Basis}\label{sec:homology}

We now use the topological properties of a real Riemann surface, recalled in Section~\ref{sec:realovals}, to construct a homology basis adapted to $\tau$.

We start with the following algebraic proposition, following the approach of Silhol \cite{Sil98}.

\begin{proposition}\label{prop:freemodule}
	Let $N$ be a free $\Z$-module of rank $n$, and let $f$ be an involution on $N$. Then, there exists a basis $(a_1, \ldots, a_d, b_1, \ldots, b_{n-d})$ of $N$ such that
	\begin{align*}
		f(a_i) &= a_i \quad \text{for } 1 \leq i \leq d,\\
		f(b_i) &= \sum_{k=1}^d h_{ki} a_k - b_i \quad \text{for } 1 \leq i \leq n-d,
	\end{align*}
	where $h_{ki} \in \{0,1\}$ are the entries of some $d \times (n-d)$ matrix $H$.
\end{proposition}

\begin{proof}
	Let $\fix(N) \coloneq \{x \in N \mid f(x) = x\}$ denote the set of fixed points of $f$ in $N$. Clearly,
	\[
	\fix(N) = \{x \in N \mid x - f(x) = 0\} = \ker(one-f).
	\]
	Let $d$ denote the rank of $\ker(one-f)$. Since $\ker(1 - f)$ is a direct factor in $N$, we have the decomposition
	\[
	N \cong \ker(one-f) \oplus N/ \ker(one-f).
	\]
	Furthermore, since $f$ is an involution, $(f \circ (one-f))(x) = f(x) - x = -(x - f(x)) = -(one-f)(x)$. Thus, the map induced by $f$ on the quotient $N/\ker(one-f)$ is $- \mathrm{Id}$. Combining a basis of $\ker(one-f)$ with a lift of a basis of $N/ \ker(one-f)$ to form a basis of $N$, we obtain the following matrix representation of $f$:
	\[
	M =
	\begin{pmatrix}
		I_d & H \\
		0 & -I_{n-d}
	\end{pmatrix}
	\]
	for some $d \times (n-d)$ matrix $H$. Here, $I$ denotes the identity matrix. By changing the basis via matrices of the form
	\[
	\begin{pmatrix}
		I_d & B \\
		0 & I_{n-d}
	\end{pmatrix},
	\]
	we transform the representation matrix of $f$ to
	\[
	\begin{pmatrix}
		I_d & B \\
		0 & I_{n-d}
	\end{pmatrix}
	\begin{pmatrix}
		I_d & H \\
		0 & -I_{n-d}
	\end{pmatrix}
	\begin{pmatrix}
		I_d & -B \\
		0 & I_{n-d}
	\end{pmatrix}
	=
	\begin{pmatrix}
		I_d & -2B + H \\
		0 & -I_{n-d}
	\end{pmatrix}.
	\]
	By choosing the entries of $B$ as $b_{ij} = \lfloor \frac{h_{ij}}{2} \rfloor$, we can ensure that the entries of the transformed matrix (which we again call $H$) lie in $\{0,1\}$.
\end{proof}

The involution $\tau$ defines a map on homology
\[
\tau : H_1(\Sigma,\Z) \to H_1(\Sigma,\Z), \quad [\phi] \mapsto [\tau(\phi)],
\]
which, for simplicity, we also denote by $\tau$. The group $H_1(\Sigma,\Z)$ is a free module of rank $2g$. Applying Proposition~\ref{prop:freemodule} with $N = H_1(\Sigma,\Z)$ and $f = \tau$, we seek to characterize the rank of the submodule of fixed points and the rank of $H$.

\begin{proposition}\label{prop:fixedpointhomology}
	The fixed point set $\fix(H_1(\Sigma,\Z)) = \{ c \in H_1(\Sigma,\Z) \mid \tau(c) = c \}$ is a submodule of rank $g$.
\end{proposition}

\begin{proof}
	Consider the $\Q$-vector space $V \coloneq \Q \otimes_{\Z} H_1(\Sigma,\Z)$. The intersection pairing defines a symplectic form on $V$. Since $\tau$ is an involution, its eigenvalues are $+1$ and $-1$, and the minimal polynomial divides $x^2 - 1 = (x-1)(x+1)$. It follows that $\tau$ is diagonalizable and $V$ decomposes as $V = E_+ \oplus E_-$, where $E_+$ and $E_-$ denote the eigenspaces for the eigenvalues $+1$ and $-1$, respectively. In particular, $E_+$ corresponds to the set of fixed points of $\tau$.
	
	The involution $\tau$ reverses orientation, and thus changes the sign of the intersection number:
	\[
	\mathrm{int}(\tau(v), \tau(w)) = - \mathrm{int}(v,w)
	\]
	for all $v, w \in V$. This means $\tau$ is antisymplectic. For any $v_1, v_2 \in E_+$, we have:
	\[
	\mathrm{int}(v_1,v_2) = - \mathrm{int}(\tau(v_1),\tau(v_2)) = - \mathrm{int}(v_1,v_2),
	\]
	which implies $\mathrm{int}(v_1,v_2) = 0$. Hence, $E_+$ is an isotropic subspace, meaning $E_+ \subset (E_+)^\perp = \{ v \in V \mid \mathrm{int}(v,w) = 0 \ \forall w \in E_+ \}$. The same holds for $E_-$. Since the symplectic form is non-degenerate, the dimension of an isotropic subspace cannot exceed half the dimension of the total space $V$. Since $V = E_+ \oplus E_-$ and $\dim(V) = 2g$, both subspaces must have dimension exactly $g$. We conclude that the rank of $\fix(H_1(\Sigma,\Z))$ equals $\dim(E_+) = g$.
\end{proof}

This decomposition leads to the following immediate consequence:

\begin{corollary}\label{cor:homologybasis1}
	There exists a basis $(a_1, \ldots, a_g, b_1, \ldots, b_g)$ of $H_1(\Sigma,\Z)$ such that $\tau(a_i) = a_i$ for $i = 1, \ldots, g$.
\end{corollary}

For our purposes, we require a basis that is also \emph{symplectic}. Consider $V_2 \coloneq H_1(\Sigma,\Z_2)$. Equipped with the intersection number mod 2, $V_2$ is a symplectic $\Z_2$-vector space. Furthermore,
\[
\mathrm{int}(\phi,\psi) \equiv \mathrm{int}(\tau(\phi),\tau(\psi)) \pmod 2
\]
holds, meaning that $\tau$ acts symplectically on $V_2$. Using the basis form from Proposition~\ref{prop:freemodule} (where $d=g$), the representation matrix is
\[
M = \begin{pmatrix}
	I_g & H \\
	0 & I_g
\end{pmatrix}.
\]
Since $M$ must preserve the symplectic form $J = \begin{psmallmatrix} 0 & I_g \\ I_g & 0 \end{psmallmatrix}$, we compute:
\[
\begin{pmatrix}
	I_g & 0 \\
	H^T & I_g
\end{pmatrix}
\begin{pmatrix}
	0 & I_g \\
	I_g & 0
\end{pmatrix}
\begin{pmatrix}
	I_g & H \\
	0 & I_g
\end{pmatrix}
=
\begin{pmatrix}
	0 & I_g \\
	I_g & 0
\end{pmatrix}.
\]
The left hand side evaluates to:
\[
\begin{pmatrix}
	0 & I_g \\
	I_g & H^T
\end{pmatrix}
\begin{pmatrix}
	I_g & H \\
	0 & I_g
\end{pmatrix}
=
\begin{pmatrix}
	0 & I_g \\
	I_g & H^T + H
\end{pmatrix}.
\]
Comparison implies $H^T + H = 0$, so $H = H^T$ (in $\Z_2$, antisymmetry is equivalent to symmetry). Thus, $H$ is symmetric. Any symmetric matrix over $\Z_2$ is characterized by its rank and its diagonal (see Corollary~\ref{cor:characterization} in Appendix~\ref{sec:bilinearforms}). We define $\mathrm{diag}(H)$ by:
\[
\mathrm{diag}(H) \coloneq
\begin{cases}
	0, & \text{if } H_{ii} = 0 \ \forall i=1,\dots,g,\\
	1, & \text{otherwise}.
\end{cases}
\]

\begin{remark}[Discrete preview]
	In the discrete theory, we consider homology on the \emph{medial graph}, which is constructed by connecting the midpoints of edges of the quad-graph (see Section~\ref{sec:Riemann}). The homology classes $a_i, b_i$ will be represented by cycles on this graph. Consequently, the intersection number $\mathrm{int}(\cdot, \cdot)$ corresponds to the algebraic sum of crossings of these paths. The explicit geometric construction of the basis below serves as the direct template for the discretization.
\end{remark}

We now provide the complete classification of the representation matrix of $\tau$ in homology, depending on the type of real Riemann surface. While the result is known (e.g., \cite{Vin93}), we follow the constructive ideas of Weichold \cite{Wei83} to ensure the method is transferable to the discrete setting.

\begin{proposition}\label{prop:realhomology}
	Let $(\Sigma,\tau)$ be a real Riemann surface of genus $g$ with $k$ real ovals. Then, there exists a symplectic basis $(a_1, \ldots, a_g, b_1, \ldots, b_g)$ of $H_1(\Sigma,\Z)$ such that
	\begin{align*}
		\tau(a_i) &= a_i \quad \text{for } 1 \leq i \leq g,\\
		\tau(b_i) &= \sum_{j=1}^g h_{ji} a_j - b_i \quad \text{for } 1 \leq i \leq g,
	\end{align*}
	where $H = (h_{ij})$ is a symmetric matrix with entries in $\{0,1\}$. Furthermore:
	\begin{enumerate}
		\item[(i)] If $\Sigma$ is dividing, then $\mathrm{diag}(H) = 0$ and $\mathrm{rank}(H) = g + 1 - k$.
		\item[(ii)] If $\Sigma$ is non-dividing and $k \neq 0$, then $\mathrm{diag}(H) = 1$ and $\mathrm{rank}(H) = g + 1 - k$.
		\item[(iii)] If $\Sigma$ is non-dividing and $k = 0$, then $\mathrm{diag}(H) = 0$ and
		\[
		\mathrm{rank}(H) =
		\begin{cases}
			g & \text{if } g \equiv 0 \pmod 2, \\
			g - 1 & \text{if } g \equiv 1 \pmod 2.
		\end{cases}
		\]
	\end{enumerate}
\end{proposition}

\begin{remark}
	According to Proposition~\ref{prop:number_ovals}, $0 \leq k \leq g+1$, so $\mathrm{rank}(H)\geq 0$. If $\Sigma$ is dividing, then $k \neq 0$, ensuring $\mathrm{rank}(H)\leq g$.
	Asking for a symplectic basis (rather than just any basis) requires careful choices of orientations and explicit calculation of intersection numbers, which we perform below.
\end{remark}

\begin{proof}
	(i) \textbf{Dividing Type.} Let $x_0, \dots, x_{k-1}$ denote the real ovals. We denote homology classes by square brackets, e.g., $[x_0]$. Since $\Sigma$ is dividing, $\Sigma \setminus \fix(\tau)$ consists of two connected components $S_1$ and $S_2$, with $\tau(S_1)=S_2$ (see Figure~\ref{fig:dividing6}).
	
	Pick a base point $P_i$ on each $x_i$ ($i = 0, \dots, k-1$). Construct paths $y_i$ from $P_0$ to $P_i$ lying entirely inside $S_1$, avoiding all other ovals and not encircling any handles of $S_1$. Then $\tau(y_i)$ lies in $S_2$ and connects $P_0$ to $P_i$. Set
	\[
	c_i \coloneq (\tau(y_i))^{-1} \circ y_i.
	\]
	Observe that $\tau(c_i) = (c_i)^{-1}$. Since $\tau$ is antisymplectic, \[\mathrm{int}([c_i], [c_j]) = - \mathrm{int}(\tau([c_i]), \tau([c_j])) = - \mathrm{int}(-[c_i], -[c_j]) = - \mathrm{int}([c_i], [c_j]).\] We can orient the ovals $x_0, \dots, x_{k-1}$ such that
	\[
	\mathrm{int}([x_i], [x_j]) = 0 = \mathrm{int}([c_i], [c_j]), \quad
	\mathrm{int}([x_i], [c_j]) = \delta_{ij}.
	\]

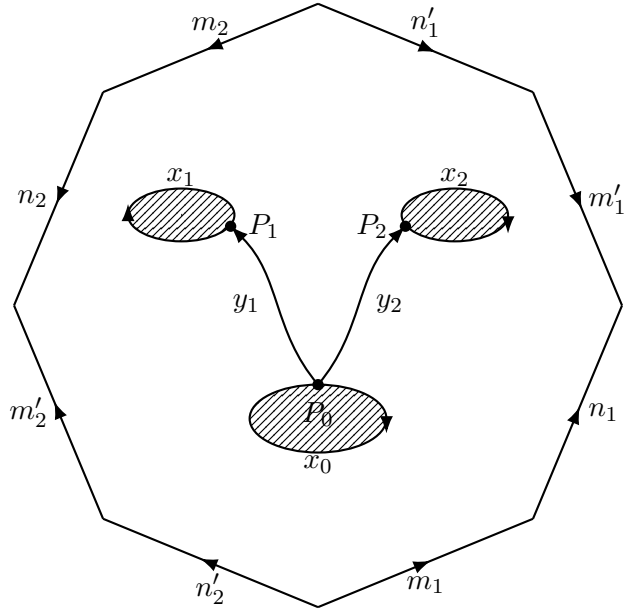
\begin{figure}[htbp]
	\centering
	\begin{tikzpicture}[
		scale=1.0,
		>=Latex,
		thick,
		midarrow/.style={decoration={markings, mark=at position 0.55 with {\arrow{>}}}, postaction={decorate}},
		revmidarrow/.style={decoration={markings, mark=at position 0.55 with {\arrow{<}}}, postaction={decorate}},
		hole/.style={fill=white, pattern=north east lines, pattern color=black},
		pathLine/.style={thick, ->},
		point/.style={fill=black, circle, inner sep=1.5pt}
		]
		\def\R{4.0}
		\foreach \i in {1,...,8} {
			\coordinate (V\i) at ({90 - (\i-1)*360/8}:\R);
		}
		
		\draw[midarrow] (V1) -- (V2) node[midway, above] {$n'_1$};
		\draw[midarrow] (V2) -- (V3) node[midway, right] {$m'_1$};
		\draw[revmidarrow] (V3) -- (V4) node[midway, right] {$n_1$};
		\draw[revmidarrow] (V4) -- (V5) node[midway, below, yshift=-1pt] {$m_1$};
		
		\draw[midarrow] (V5) -- (V6) node[midway, below] {$n'_2$};
		\draw[midarrow] (V6) -- (V7) node[midway, left] {$m'_2$};
		\draw[revmidarrow] (V7) -- (V8) node[midway, left] {$n_2$};
		\draw[revmidarrow] (V8) -- (V1) node[midway, above, yshift=1pt] {$m_2$};
		
\coordinate (C0) at (0, -1.5);
\draw[hole] ($(C0)+(0.9,0)$) arc (0:360:0.9 and 0.45);
\node[below, yshift=-10pt] at (C0) {$x_0$};
\draw[->] (0.9,-1.6)--(0.9,-1.7);

\coordinate (C1) at (-1.8, 1.2);
\draw[hole] ($(C1)+(0.7,0)$) arc (0:360:0.7 and 0.35);
\node[above, yshift=7pt] at (C1) {$x_1$};
\draw[->] (C1)--(C1)+(-0.7,0.15);

\coordinate (C2) at (1.8, 1.2);
\draw[hole] ($(C2)+(0.7,0)$) arc (0:360:0.7 and 0.35);
\node[above, yshift=7pt] at (C2) {$x_2$};
\draw[->] (2.5,1.0)--(2.5,0.95);
		
		\coordinate (P0) at ($(C0)+(0, 0.45)$);
		\node[point, label={below:$P_0$}] at (P0) {};
		
		\coordinate (P1) at ($(C1)+(0.65, -0.15)$);
		\node[point, label={right:$P_1$}] at (P1) {};
		
		\coordinate (P2) at ($(C2)+(-0.65, -0.15)$);
		\node[point, label={left:$P_2$}] at (P2) {};
		
		\draw[pathLine] (P0) to[out=130, in=-45] node[midway, left, xshift=-2pt] {$y_1$} (P1);
		
		\draw[pathLine] (P0) to[out=50, in=225] node[midway, right, xshift=2pt] {$y_2$} (P2);
		
	\end{tikzpicture}
	\caption{Fundamental $4g^{\prime}$-gon $F_{g^{\prime}}$ of the component $S_1$ (here $g^{\prime}=2$) bounded by real ovals $x_0, x_1, x_2$. The paths $y_i$ connect the base point $P_0$ on $x_0$ to $P_i$ on $x_i$.}
	\label{fig:fundamental}
\end{figure}
	
	Let $g^{\prime}$ be the genus of $S_1$. Consider a fundamental $4g^{\prime}$-gon representation of $S_1$ with the real ovals in its interior (Figure~\ref{fig:fundamental}). Let the canonical edges be $m_1, \dots, m_{g^{\prime}}$ and $n_1, \dots, n_{g^{\prime}}$, satisfying
	\[
	\mathrm{int}([n_i], [n_j]) = 0 = \mathrm{int}([m_i], [m_j]), \quad
	\mathrm{int}([m_i], [n_j]) = \delta_{ij}.
	\]
	These curves lie in $S_1$ and avoid both the ovals and the curves $c_i$. Their images $\tau(m_i), \tau(n_i)$ lie in $S_2$. Define:
	\begin{align*}
		a^{\prime}_i &\coloneq [m_i] + \tau([m_i]), & b^{\prime}_i &\coloneq [n_i],\\
		a^{\prime \prime}_i &\coloneq [n_i] + \tau([n_i]), & b^{\prime \prime}_i &\coloneq \tau([m_i]).
	\end{align*}
	
	We claim that the ordered set
	\[
	\left( a^{\prime}_1, a^{\prime \prime}_1, \ldots, a^{\prime}_{g^{\prime}}, a^{\prime \prime}_{g^{\prime}}, [x_1], \dots, [x_{k-1}], b^{\prime}_1, b^{\prime \prime}_1, \ldots, b^{\prime}_{g^{\prime}}, b^{\prime \prime}_{g^{\prime}}, [c_1], \dots, [c_{k-1}] \right)
	\]
	forms a symplectic basis of $H_1(\Sigma, \Z)$. From the proof of Proposition~\ref{prop:number_ovals}~(i), we know $g = 2g' + k - 1$, so the number of elements ($2(2g' + k - 1) = 2g$) is correct. We check the intersection numbers.
	
	First, $\tau(a^{\prime}_i)=a^{\prime}_i$ and $\tau(a^{\prime\prime}_i)=a^{\prime\prime}_i$. Since $\tau$ is antisymplectic on $E_+$,
	\[
	\mathrm{int}(a^{\prime}_i, a^{\prime}_j) = \mathrm{int}(a^{\prime\prime}_i, a^{\prime\prime}_j) = \mathrm{int}(a^{\prime}_i, a^{\prime\prime}_j) = 0.
	\]
	Since $n_i \subset S_1$ and $\tau(m_j) \subset S_2$ are disjoint, and the canonical curves in $S_1$ satisfy standard relations,
	\[
	\mathrm{int}(b^{\prime}_i, b^{\prime}_j) = \mathrm{int}(b^{\prime\prime}_i, b^{\prime\prime}_j) = \mathrm{int}(b^{\prime}_i, b^{\prime\prime}_j) = 0.
	\]
	The remaining cross-intersections are:
	\begin{align*}
		\mathrm{int}(a^{\prime}_i, b^{\prime}_j) &= \mathrm{int}([m_i]+\tau([m_i]), [n_j]) = \mathrm{int}([m_i], [n_j]) = \delta_{ij},\\
		\mathrm{int}(a^{\prime}_i, b^{\prime \prime}_j) &= \mathrm{int}([m_i]+\tau([m_i]), \tau([m_j])) = \mathrm{int}(\tau([m_i]),\tau([m_j])) = 0, \\
		\mathrm{int}(a^{\prime\prime}_i, b^{\prime}_j) &= \mathrm{int}([n_i]+\tau([n_i]), [n_j]) = \mathrm{int}([n_i], [n_j]) = 0,\\
		\mathrm{int}(a^{\prime\prime}_i, b^{\prime\prime}_j) &= \mathrm{int}([n_i]+\tau([n_i]), \tau([m_j])) = \mathrm{int}(\tau([n_i]), \tau([m_j])) = -\mathrm{int}([n_i], [m_j]) = \delta_{ij}.
	\end{align*}
	
	The action of $\tau$ on this basis is:
	\begin{align*}
		\tau(a^{\prime}_i) &= a^{\prime}_i, & \tau(a^{\prime\prime}_i) &= a^{\prime\prime}_i, & \tau([x_i]) &= [x_i], \\
		\tau(b^{\prime}_i) &= \tau([n_i]) = a^{\prime\prime}_i - b^{\prime}_i, & \tau(b^{\prime\prime}_i) &= [m_i] = a^{\prime}_i - b^{\prime\prime}_i, & \tau([c_j]) &= -[c_j].
	\end{align*}
	
	The matrix $H$ has $g'$ blocks of $\begin{psmallmatrix} 0 & 1 \\ 1 & 0 \end{psmallmatrix}$ in the diagonal, followed by zeros. Thus, $\mathrm{diag}(H) = 0$ and $\mathrm{rank}(H) = 2g^{\prime} = g + 1 - k$.
	
(ii) \textbf{Non-Dividing Type ($k \neq 0$).} The ovals do not separate $\Sigma$. Following \cite{Kra82}, we construct additional cuts. Since $\Sigma/\tau$ is non-orientable, choose an orientation-reversing loop $\gamma$ on the quotient. Its lift is a curve $r_k = \tau(\gamma) \circ \gamma$ from $P$ to $\tau(P)$, so $\tau([r_k]) = [r_k]$.
Iterate this construction with curves $r_k, \dots, r_l$ until $x_0, \ldots, x_{k-1}, r_k, \ldots, r_l$ divide $\Sigma$ into two components $S_1$ and $S_2$.
Since these $l+1$ curves form the boundary of $S_1$ (which has genus $g'$), the Euler characteristic implies $g = 2g' + l$. Since $g' \geq 0$, it follows that $l \leq g$.

Construct paths $y_i$ from $P_0$ to $P_i$ (where $P_i \in x_i$ or $P_i \in r_i$) inside $S_1$.
For real ovals ($i < k$), set $c_i \coloneq \tau(y_i) \circ y_i$. Orient them such that $\mathrm{int}([x_i], [c_j]) = \delta_{ij}$.
For the auxiliary curves $r_i$ ($i \ge k$), define $c_i \coloneq (\tau(y_i))^{-1} \circ \tilde{r}_i \circ y_i$, where $\tilde{r}_i$ is the directed arc of the curve $r_i$ connecting $P_i$ to $\tau(P_i)$. The path $\tau(c_i)$ starts at $P_0$, goes to $\tau(P_i)$, follows $r_i$ to $P_i$, and returns to $P_0$. We can deform the paths such that $c_i$ and $\tau(c_i)$ intersect only at $P_0$ and $P_i$. Crucially, for these auxiliary curves, the geometry forces non-zero intersection with their image (see Figure~\ref{fig:non-dividing}). Orient $r_i$ such that $\mathrm{int}([\tau(c_j)],[c_i]) = \delta_{ij}$.

\begin{figure}[htbp]
	\centering
	\subfloat[Genus 1 example.]{\includegraphics[width=0.45\textwidth]{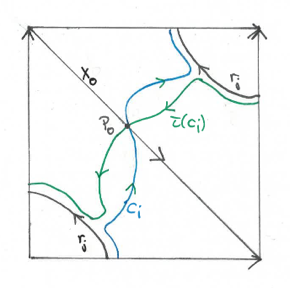}}
	\hfill
	\subfloat[General construction.]{\includegraphics[width=0.45\textwidth]{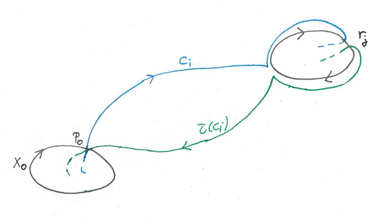}}
	\caption{Construction of curves $c_i$ (blue) and $\tau(c_i)$ (green) for a non-dividing surface.}
	\label{fig:non-dividing}
\end{figure}

Using the fundamental polygon of $S_1$ to define $m_i, n_i$ as before, we define the additional basis elements:
\begin{align*}
	a_i &\coloneq [c_{i+k-1}] + \tau([c_{i+k-1}]), & b_i &\coloneq [c_{i+k-1}], & i &= 1, \ldots, l-k+1;\\
	a^{\prime}_i &\coloneq [m_i] + \tau([m_i]), & b^{\prime}_i &\coloneq [m_i], & i &= 1, \ldots, g';\\
	a^{\prime\prime}_i &\coloneq [n_i] + \tau([n_i]), & b^{\prime\prime}_i &\coloneq [n_i], & i &= 1, \ldots, g'.
\end{align*}

We claim that the ordered set
\begin{eqnarray*}
	\left(
	a_1, \ldots, a_{l-k+1}, 
	a^{\prime}_1, a^{\prime \prime}_1, \ldots, a^{\prime}_{g^{\prime}}, a^{\prime \prime}_{g^{\prime}}, 
	[x_1], \dots, [x_{k-1}], 
	b_1, \ldots, b_{l-k+1}, 
	b^{\prime}_1, b^{\prime \prime}_1, \ldots, b^{\prime}_{g^{\prime}}, b^{\prime \prime}_{g^{\prime}}, 
	[c_1], \dots, [c_{k-1}]
	\right)
\end{eqnarray*}
forms a symplectic basis of $H_1(\Sigma, \Z)$. The intersection numbers for $a'_i, a''_i, b'_i, b''_i$ and the ovals $[x_n]$ are as in case (i). By construction, the paths $y_i$ are chosen to be disjoint from each other and from the previously defined curves. Consequently, the cycles $[c_{i+k-1}]$ do not intersect $[c_{j+k-1}]$ for $i \neq j$, nor do they intersect the interior cycles of $S_1$ defined by the fundamental polygon. For the new elements $a_i, b_i$:
\begin{align*}
	\mathrm{int}(a_i, b_j) &= \mathrm{int}(\tau([c_{i+k-1}]), [c_{j+k-1}]) = \delta_{ij}, \\
	\mathrm{int}(a_i, a_j) &= \mathrm{int}([c_{i+k-1}] + \tau([c_{i+k-1}]), [c_{j+k-1}] + \tau([c_{j+k-1}])) \\
	&= \mathrm{int}([c_{i+k-1}], \tau([c_{j+k-1}])) + \mathrm{int}(\tau([c_{i+k-1}]), [c_{j+k-1}]) \\
	&= -\delta_{ji} + \delta_{ij} = 0, \\
	\mathrm{int}(b_i, b_j) &= \mathrm{int}([c_{i+k-1}], [c_{j+k-1}]) = 0.
\end{align*}
The action of $\tau$ on the basis elements inherited from case (i) remains unchanged. For the newly added elements, we have $\tau(a_i) = a_i$ and $\tau(b_i) = a_i - b_i$. Consequently, the matrix $H$ has $l-k+1$ ones on the diagonal followed by $g'$ blocks of $\begin{psmallmatrix} 0 & 1 \\ 1 & 0 \end{psmallmatrix}$. Thus, $\mathrm{diag}(H) = 1$ and $\mathrm{rank}(H) = (l-k+1) + 2g' = g + 1 - k$.

	(iii) \textbf{Non-Dividing Type ($k = 0$).} We use auxiliary curves $r_0, \ldots, r_{l-1}$ to divide $\Sigma$. These are constructed iteratively exactly as in case (ii) by lifting orientation-reversing loops from the quotient surface. Since there are no real ovals, these curves form the sole boundary of the components $S_1$ and $S_2$, implying the homology relation $[r_0] = \sum_{i=1}^{l-1} [r_i]$ (with appropriate orientation).
	
	We construct $c_i$ via paths $y_i$ from $P_0 \in r_0$ to $P_i \in r_i$ as in (ii), defined by $c_i \coloneq (\tau(y_i))^{-1} \circ \tilde{r}_i \circ y_i$. The image cycle $\tau(c_i)$ geometrically traces the same path segments. However, by slightly deforming $\tau(c_i)$ -- specifically, by shifting the path along the boundary curves $r_0$ and $r_i$ so that it intersects $c_i$ transversely only at the base points $P_0$ and $P_i$ with opposite signs -- we ensure that the algebraic intersection number vanishes: $\mathrm{int}([c_i], \tau([c_i])) = 0$ (see Figure~\ref{fig:non_dividing_c}).
	
	\begin{figure}[htbp]
		\centering
		\includegraphics[width=0.5\textwidth]{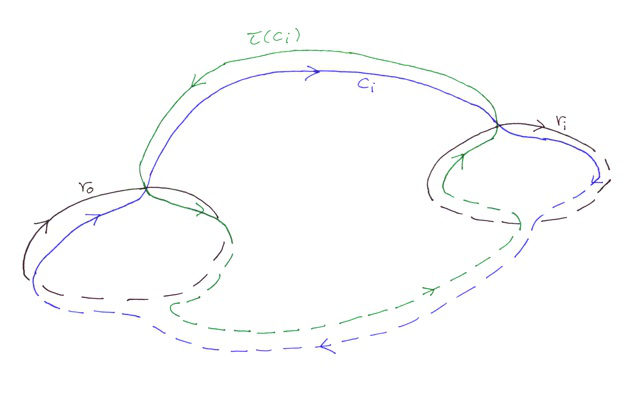}
		\caption{Construction of the homology basis for the case $k = 0$.}
		\label{fig:non_dividing_c}
	\end{figure}

Define the basis elements, using the canonical cycles $m_j, n_j$ of the fundamental polygon of $S_1$:
\begin{align*}
	a_i &\coloneq [r_i], & b_i &\coloneq -[r_i] + [c_i], & i &= 1, \ldots, l - 1; \\
	a^{\prime}_j &\coloneq [m_j] + \tau([m_j]), & b^{\prime}_j &\coloneq [n_j], & j &= 1, \ldots, g'; \\
	a^{\prime\prime}_j &\coloneq [n_j] + \tau([n_j]), & b^{\prime\prime}_j &\coloneq \tau([m_j]), & j &= 1, \ldots, g'.
\end{align*}

Analogously to our observation in (ii), we have $l=1+g-2g^{\prime}$, so $2(2g'+l-1)=2g$ and the number of elements is correct. We claim that the ordered set
\[
\left(
a_1, \ldots, a_{l-1}, 
a^{\prime}_1, a^{\prime \prime}_1, \ldots, a^{\prime}_{g^{\prime}}, a^{\prime \prime}_{g^{\prime}},
b_1, \ldots, b_{l-1}, 
b^{\prime}_1, b^{\prime \prime}_1, \ldots, b^{\prime}_{g^{\prime}}, b^{\prime \prime}_{g^{\prime}}
\right)
\]
forms a symplectic basis of $H_1(\Sigma, \Z)$. Checking intersection numbers:
\begin{align*}
	\mathrm{int}(a_i, a_j) &= \mathrm{int}([r_i], [r_j]) = 0, \\
	\mathrm{int}(b_i, b_j) &= \mathrm{int}(-[r_i]+[c_i], -[r_j]+[c_j]) = -\mathrm{int}([r_i],[c_j]) - \mathrm{int}([c_i],[r_j]) = -\delta_{ij} + \delta_{ji} = 0, \\
	\mathrm{int}(a_i, b_j) &= \mathrm{int}([r_i], -[r_j] + [c_j]) = \delta_{ij}.
\end{align*}

For the action of $\tau$: $\tau(a_i) = a_i$. For $b_i$, note that $[c_i] + \tau([c_i]) = [r_0] + [r_i]$. Thus:
\[
\tau(b_i) = -[r_i] + \tau([c_i]) = [r_0] - [c_i] = \sum_{j=1}^{l-1} a_j - (b_i + a_i) = \sum_{j \neq i} a_j - b_i.
\]
The matrix $H$ contains a dense block $D$ of size $(l-1) \times (l-1)$ with zeros on the diagonal and ones elsewhere, plus the usual $g'$ blocks of of $\begin{psmallmatrix} 0 & 1 \\ 1 & 0 \end{psmallmatrix}$.
Thus $\mathrm{diag}(H) = 0$. The rank analysis of $D$ modulo 2 (summing columns) shows $\mathrm{rank}(H) = g$ if $g$ is even, and $g-1$ if $g$ is odd.
\end{proof}

\begin{remark}
	It is important to note that our proofs of Propositions~\ref{prop:number_ovals} and~\ref{prop:realhomology} do not make use of the complex structure of the Riemann surface $\Sigma$. They apply without modification to any closed topological orientable surface.
	This topological nature is what allows us to transfer these results directly to the discrete setting, justifying the elementary proofs provided here. Note that other proofs, e.g., in \cite{GH81}, rely on algebraic geometry and are not suitable for discretization.
\end{remark}

\subsection{Period matrix}\label{sec:periodmatrix}

Using the symplectic homology basis adapted to the involution $\tau$ (Proposition~\ref{prop:realhomology}), we now derive the structural constraints on the period matrix of a real Riemann surface. While these results are classical \cite{BBEIM94,Vin93}, we present the derivation here to highlight the algebraic mechanism that will be mirrored exactly in the discrete setting.

\begin{definition}
	Let $\{a_1,\ldots,a_g,b_1,\ldots,b_g\}$ be a symplectic homology basis of a Riemann surface $\Sigma$. Let $\{\omega_1, \ldots, \omega_g\}$ be the dual basis of holomorphic differentials normalized by $\int_{a_i} \omega_j=\delta_{ij}$. The \textit{period matrix} $\Pi_{\Sigma}$ is the $g \times g$ matrix with entries
	\[
	(\Pi_{\Sigma})_{ij} \coloneq \int_{b_i} \omega_j.
	\]
\end{definition}

\begin{remark}[Discrete preview]
	In Section~\ref{sec:discrete_real}, the holomorphic differentials $\omega_j$ will be replaced by discrete holomorphic forms living on the edges of the medial graph (constructed by connecting the midpoints of the edges of the underlying bipartite quad-graph). The bipartite nature of the graph leads to a doubling of the dimension of the space of holomorphic forms, resulting in distinct \emph{black} and \emph{white} periods. While a specific normalization -- taking the arithmetic mean of these periods -- restores the correspondence to the classical period matrix $\Pi_{\Sigma}$, the study of the full \emph{complete discrete period matrix} remains of independent theoretical interest.
\end{remark}

The period matrix depends on the choice of the homology basis. Since any two symplectic bases are related by a symplectic transformation, the period matrices transform accordingly (via the action of the modular group). It is a fundamental result that $\Pi_{\Sigma}$ is symmetric and has a positive definite imaginary part. For real Riemann surfaces, the real part is constrained by the topology of the involution.

\begin{proposition}\label{prop:matrix_smooth}
	Let $(\Sigma, \tau)$ be a real Riemann surface. With respect to the homology basis constructed in Proposition~\ref{prop:realhomology}, the period matrix decomposes as
	\[
	\Pi_{\Sigma} = \frac{1}{2} H + i T,
	\]
	where $T \in GL(g, \R)$ is a real matrix and $H \in M(g, \Z_2)$ is the binary matrix describing the action of $\tau$ on the homology basis (from Proposition~\ref{prop:realhomology}).
\end{proposition}

\begin{proof}
	From Proposition~\ref{prop:realhomology}, our basis satisfies $\tau(a_i) = a_i$ and $\tau(b_i) = \sum_{k=1}^g h_{ki} a_k - b_i$.
	
Consider the differential forms $\omega^*_j \coloneq \overline{\tau^*\omega_j}$, defined via the pullback $\tau^*\omega_j(p) = \omega_j(\tau(p))$. Since $\tau$ is antiholomorphic, $\tau^*\omega_j$ is antiholomorphic, and thus the conjugate $\omega^*_j$ is a holomorphic differential. We check its periods along the $a$-cycles. Since $\tau(a_i) = a_i$ (preserving orientation in the integral due to the double conjugation of map and form):
\[
\int_{a_i} \omega^*_j = \int_{a_i} \overline{\tau^* \omega_j} = \overline{\int_{\tau(a_i)} \omega_j} = \overline{\int_{a_i} \omega_j} = \overline{\delta_{ij}} = \delta_{ij}.
\]
Since a holomorphic differential is uniquely determined by its $a$-periods, we conclude that $\omega^*_j = \omega_j$ for all $j=1,\ldots,g$.
	
	We now calculate the entries of the period matrix. Using $\omega_j = \overline{\tau^*\omega_j}$:
	\begin{align*}
		(\Pi_{\Sigma})_{ij} &= \int_{b_i} \omega_j = \int_{b_i} \overline{\tau^*\omega_j} = \overline{\int_{\tau(b_i)} \omega_j}.
	\end{align*}
	Substituting the homology relation $\tau(b_i) = \sum_{k=1}^g h_{ki} a_k - b_i$:
	\begin{align*}
		(\Pi_{\Sigma})_{ij} &= \overline{\int_{\sum_{k} h_{ki} a_k - b_i} \omega_j} 
		= \sum_{k=1}^g h_{ki} \overline{\int_{a_k} \omega_j} - \overline{\int_{b_i} \omega_j} \\
		&= \sum_{k=1}^g h_{ki} \delta_{kj} - \overline{(\Pi_{\Sigma})_{ij}} \\
		&= h_{ji} - \overline{(\Pi_{\Sigma})_{ij}}.
	\end{align*}
	This implies $(\Pi_{\Sigma})_{ij} + \overline{(\Pi_{\Sigma})_{ij}} = h_{ji}$. Since $\Pi_{\Sigma}$ is symmetric, we have $2 \Re((\Pi_{\Sigma})_{ij}) = h_{ji} = h_{ij}$. Thus, the real part of the period matrix is exactly $\frac{1}{2}H$.
\end{proof}

\begin{remark}[Discrete preview]
	This result is of central importance for the discrete theory. Since the discrete period matrix will be defined via linear operations that respect the involution symmetry, we will find that the discrete period matrix satisfies the exact same relation $\Pi = \frac{1}{2} H + i T$. This ensures that discrete M-curves also possess purely imaginary period matrices.
\end{remark}

\begin{corollary}\label{cor:m-curve}
	The period matrix of an M-curve $(\Sigma,\tau)$ is purely imaginary (with respect to the homology basis of Proposition~\ref{prop:realhomology}).
\end{corollary}

\begin{proof}
	An M-curve has $k=g+1$ real ovals. By Proposition~\ref{prop:number_ovals}, it is of dividing type. Proposition~\ref{prop:realhomology} then implies $\mathrm{rank}(H) = g+1 - k = 0$. Since $H$ is the zero matrix, Proposition~\ref{prop:matrix_smooth} yields $\Pi_{\Sigma} = iT$ with $T \in GL(g,\R)$.
\end{proof}

Finally, we note that the converse holds, a result due to Silhol \cite{Sil01}, which provides an algebraic characterization of M-curves.

\begin{proposition}\label{prop:Silhol_converse}
	Let $\Sigma$ be a Riemann surface of genus $g$ with period matrix $\Pi_{\Sigma}$. If $\Re(\Pi_{\Sigma}) = 0$, then there exists an antiholomorphic involution $\tau$ such that $(\Sigma,\tau)$ is an M-curve.
\end{proposition}


\section{Discrete Riemann surfaces and their period matrices}\label{sec:basic}

In this section, we provide a concise summary of the linear theory of discrete Riemann surfaces based on general quad-graphs. The theory was originally developed by Mercat \cite{Me01} on a purely combinatorial level. Later, Bobenko and the second author \cite{BoG17} established a rigorous connection to the classical theory of Riemann surfaces by constructing a complex atlas of local charts, allowing for a parallel development of the discrete and continuous theories.
In this work, however, we restrict ourselves primarily to the combinatorial framework. As we will see, this level of abstraction is sufficient to define and analyze discrete real Riemann surfaces and their period matrices.

\subsection{Discrete Riemann surfaces}\label{sec:Riemann}

We replace the smooth Riemann surface with a cellular decomposition of a compact oriented surface $\Sigma$.

\begin{definition}
	A \textit{discrete Riemann surface}, denoted by $(\Sigma,\Lambda)$, consists of a compact oriented surface $\Sigma$ and a finite decomposition $\Lambda$ of $\Sigma$ into quadrilaterals $F(\Lambda)$, satisfying the following properties:
	\begin{itemize}
		\item The decomposition is a \textit{quad-graph}, meaning all its faces are quadrilaterals.
		\item It is \textit{strongly regular}, meaning that two different faces are either disjoint, share exactly one vertex, or share exactly one edge.
		\item The graph $(V(\Lambda),E(\Lambda))$ is bipartite. We fix a coloring of the vertices $V(\Lambda) = V(\Gamma) \sqcup V(\Gamma^*)$ into black vertices $V(\Gamma)$ and white vertices $V(\Gamma^*)$ such that each edge connects a black and a white vertex.
	\end{itemize}
\end{definition}

The diagonals of the quadrilaterals in $F(\Lambda)$ give rise to two connected graphs:
\begin{itemize}
	\item The \textit{black graph} $\Gamma$, whose vertices are the black vertices of $\Lambda$ and whose edges connect black vertices sharing a face in $\Lambda$.
	\item The \textit{white graph} $\Gamma^*$, whose vertices are the white vertices of $\Lambda$ and whose edges connect white vertices sharing a face in $\Lambda$.
\end{itemize}
The dual cell decomposition of $\Lambda$ is denoted by $\Diamond$.

To define a conformal structure, we assign a geometric weight to each face.

\begin{definition}
	A \textit{discrete complex structure} on $(\Sigma, \Lambda)$ is a map $\rho: F(\Lambda) \to \{z \in \mathds{C} \mid \mathrm{Re}(z) > 0\}$, assigning a coefficient $\rho_Q$ to each quadrilateral $Q$.
\end{definition}

Although we define $\rho$ abstractly, its geometric meaning justifies the term complex structure. As shown in \cite{BoG17}, the values $\rho_Q$ allow one to construct local charts where each abstract quadrilateral $Q$ is realized as a geometric quadrilateral in $\mathds{C}$ (see Figure~\ref{fig:quadrilateral}).

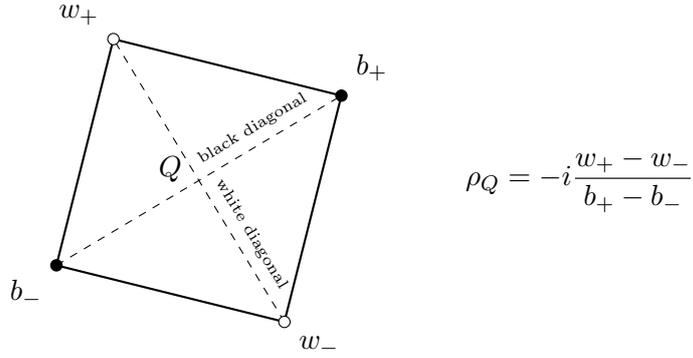
\begin{figure}[htbp]
	\centering
	\begin{tikzpicture}[scale=1.5,
		black_vertex/.style={circle,draw=black,fill=black,inner sep=0pt,minimum size=1.5mm},
		white_vertex/.style={circle,draw=black,fill=white,inner sep=0pt,minimum size=1.5mm}]
		
		\coordinate (b_minus) at (0,0);
		\coordinate (w_minus) at (2,-0.5);
		\coordinate (b_plus) at (2.5,1.5);
		\coordinate (w_plus) at (0.5,2);
		
		\draw[thick] (b_minus) -- (w_minus) -- (b_plus) -- (w_plus) -- cycle;

		\draw[dashed] (b_minus) -- (b_plus) node[pos=0.72, sloped, above, font=\tiny] {black diagonal};
		\draw[dashed] (w_minus) -- (w_plus) node[pos=0.28, sloped, above, font=\tiny] {white diagonal};
		
		\node[black_vertex, label=below left:$b_-$] at (b_minus) {};
		\node[white_vertex, label=below right:$w_-$] at (w_minus) {};
		\node[black_vertex, label=above right:$b_+$] at (b_plus) {};
		\node[white_vertex, label=above left:$w_+$] at (w_plus) {};
		
		\node at (1, 0.85) {$Q$};
		
		\node[anchor=west] at (3.5, 0.75) {$\displaystyle \rho_Q = -i\frac{w_+ - w_-}{b_+ - b_-}$};
		
	\end{tikzpicture}
	\caption{A quadrilateral $Q \in F(\Lambda)$. The diagonals correspond to the edges of the black and white graphs, respectively. The discrete complex structure $\rho_Q$ is defined by the quotient of the diagonals multiplied by $-i$.}
	\label{fig:quadrilateral}
\end{figure}

\begin{remark}
	In such a local chart, if the vertices $b_-, w_-, b_+, w_+$ are ordered counterclockwise, the coefficient is given by the cross-ratio of the diagonals:
	\[ \rho_Q = -i\frac{w_+ - w_-}{b_+ - b_-}. \]
	The condition $\mathrm{Re}(\rho_Q) > 0$ ensures that the quadrilateral is positively oriented and non-degenerate.
	If the diagonals are orthogonal, $\rho_Q$ is a positive real number. In this case, we call the discretization \textit{orthodiagonal}.
\end{remark}

\subsection{Discrete holomorphic functions and differentials}\label{sec:differential}

We begin by defining holomorphicity for functions defined on the vertices of the quad-graph. This definition relies on the geometric structure imposed by the coefficients $\rho_Q$.

\begin{definition}
	A function $f: V(\Lambda) \to \mathds{C}$ is called \textit{discrete holomorphic} if, for every quadrilateral $Q \in F(\Lambda)$, it satisfies the \textit{discrete Cauchy-Riemann equation}:
	\[ f(w_+) - f(w_-) = i \rho_Q (f(b_+) - f(b_-)), \]
	where $b_-, w_-, b_+, w_+$ are the vertices of $Q$ in counterclockwise order, starting with a black vertex (see Figure~\ref{fig:quadrilateral}).
\end{definition}

Geometrically, this condition mimics the continuous Cauchy-Riemann property, where the derivative in one direction determines the derivative in the orthogonal direction via multiplication by $i$. Here, the partial derivatives are replaced by finite differences along the diagonals, and we have to take the geometry of the quadrilateral given by the factor $\rho_Q$ into account. This approach represents a \textit{linear discretization}, compared to other discretizations of complex analysis relying on non-linear equations such as circle patterns.

To define discrete holomorphic \textit{differentials}, we introduce the medial graph $X$, which serves as the domain for discrete one-forms.

\begin{definition}
	The \textit{medial graph} $X$ of $\Lambda$ is defined combinatorially as follows:
	\begin{itemize}
		\item The vertices $V(X)$ correspond to the edges $E(\Lambda)$. We identify these vertices with the midpoints of the edges of $\Lambda$.
		\item Two vertices in $V(X)$ are connected by an edge in $E(X)$ if the corresponding edges in $\Lambda$ share a vertex and belong to the same quadrilateral.
	\end{itemize}
\end{definition}

The faces of $X$ correspond bijectively to the union $V(\Lambda) \cup F(\Lambda)$. A face $F_Q \in F(X)$ (for $Q \in F(\Lambda)$) is formed by the four edges of $X$ within the quadrilateral $Q$. A face $F_v \in F(X)$ (for $v \in V(\Lambda)$) is formed by the edges of $X$ surrounding the vertex $v$.

\begin{figure}[htbp]
	\centering
	\begin{tikzpicture}[scale=1.4,
		black_vertex/.style={circle,draw=black,fill=black,inner sep=0pt,minimum size=1.5mm},
		white_vertex/.style={circle,draw=black,fill=white,inner sep=0pt,minimum size=1.5mm},
		medial_vertex/.style={circle,draw=black,fill=gray,inner sep=0pt,minimum size=1mm},
		black_edge_style/.style={thick, dotted, black},
		white_edge_style/.style={thick, dashed, black}
		]
		
		\coordinate (c) at (0,0); 
		\coordinate (n) at (0,2);
		\coordinate (e) at (2.2,0.2);
		\coordinate (s) at (-0.2,-2);
		\coordinate (w) at (-2,-0.2);
		\coordinate (ne) at (2, 2.2);
		\coordinate (se) at (1.8, -1.8);
		\coordinate (sw) at (-2.2, -1.8);
		\coordinate (nw) at (-1.8, 1.8);
		
		\draw[thin, gray] (c) -- (n); \draw[thin, gray] (c) -- (e);
		\draw[thin, gray] (c) -- (s); \draw[thin, gray] (c) -- (w);
		\draw[thin, gray] (n) -- (ne) -- (e); \draw[thin, gray] (e) -- (se) -- (s);
		\draw[thin, gray] (s) -- (sw) -- (w); \draw[thin, gray] (w) -- (nw) -- (n);
		
		\coordinate (m_cn) at ($(c)!0.5!(n)$); \coordinate (m_ce) at ($(c)!0.5!(e)$);
		\coordinate (m_cs) at ($(c)!0.5!(s)$); \coordinate (m_cw) at ($(c)!0.5!(w)$);
		\coordinate (m_n_ne) at ($(n)!0.5!(ne)$); \coordinate (m_e_ne) at ($(e)!0.5!(ne)$);
		\coordinate (m_e_se) at ($(e)!0.5!(se)$); \coordinate (m_s_se) at ($(s)!0.5!(se)$);
		\coordinate (m_s_sw) at ($(s)!0.5!(sw)$); \coordinate (m_w_sw) at ($(w)!0.5!(sw)$);
		\coordinate (m_w_nw) at ($(w)!0.5!(nw)$); \coordinate (m_n_nw) at ($(n)!0.5!(nw)$);
		
		\draw[white_edge_style] (m_cn) -- (m_ce) -- (m_cs) -- (m_cw) -- cycle;
		\draw[black_edge_style] (m_cn) -- (m_n_ne); \draw[white_edge_style] (m_n_ne) -- (m_e_ne); \draw[black_edge_style] (m_e_ne) -- (m_ce);
		\draw[black_edge_style] (m_ce) -- (m_e_se); \draw[white_edge_style] (m_e_se) -- (m_s_se); \draw[black_edge_style] (m_s_se) -- (m_cs);
		\draw[black_edge_style] (m_cs) -- (m_s_sw); \draw[white_edge_style] (m_s_sw) -- (m_w_sw); \draw[black_edge_style] (m_w_sw) -- (m_cw);
		\draw[black_edge_style] (m_cw) -- (m_w_nw); \draw[white_edge_style] (m_w_nw) -- (m_n_nw); \draw[black_edge_style] (m_n_nw) -- (m_cn);
		
		\node[black_vertex, label={[label distance=-2pt]below left:$v$}] at (c) {};
		\node[white_vertex] at (n) {}; \node[white_vertex] at (e) {}; \node[white_vertex] at (s) {}; \node[white_vertex] at (w) {};
		\node[black_vertex] at (ne) {}; \node[black_vertex] at (se) {}; \node[black_vertex] at (sw) {}; \node[black_vertex] at (nw) {};
		\foreach \pt in {m_cn, m_ce, m_cs, m_cw, m_n_ne, m_e_ne, m_e_se, m_s_se, m_s_sw, m_w_sw, m_w_nw, m_n_nw} \node[medial_vertex] at (\pt) {};
		
		\node at (1.2, 1.2) {$Q$}; \node at (0.9, 0.9) {$F_Q$}; \node at (-0.3, 0.4) {$F_v$};
		
		\node[anchor=west, scale=0.8] at (2.5, 1) {\textbf{Medial Edges:}};
		\draw[black_edge_style] (2.6, 0.6) -- (3.6, 0.6) node[right, black] {Black edge (dotted)};
		\draw[white_edge_style] (2.6, 0.2) -- (3.6, 0.2) node[right, black] {White edge (dashed)};
	\end{tikzpicture}
	\caption{The medial graph $X$ (gray vertices) on a patch of $\Lambda$. Faces of $X$ are either inside a quadrilateral ($F_Q$) or around a vertex ($F_v$).}
	\label{fig:medial_graph}
\end{figure}
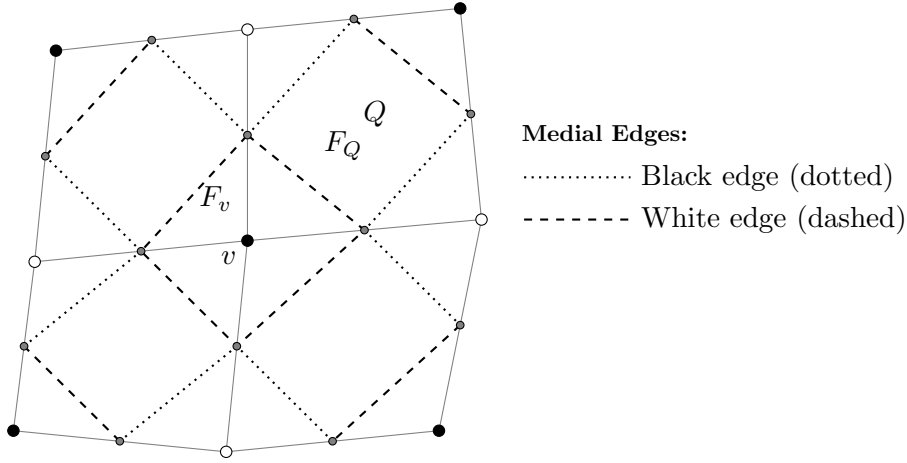

\begin{definition}
	A \textit{discrete one-form} (or discrete differential) $\omega$ is a complex-valued function defined on the directed edges of $X$, satisfying $\omega(-e) = -\omega(e)$.
\end{definition}

We focus on a specific class of one-forms that respects the geometry of the quadrilaterals.

\begin{definition}
	A discrete differential $\omega$ is said to be of \textit{type $\Diamond$} if, for every quadrilateral $Q \in F(\Lambda)$, the values of $\omega$ on the two opposite black edges of the face $F_Q$ are equal, and similarly for the two opposite white edges.
\end{definition}

\begin{remark}
	The terminology is motivated by the geometric realization. If we map the vertices of a quadrilateral $Q$ to the complex plane, the corresponding face $F_Q$ of the medial graph (connecting the midpoints of the edges of $Q$) forms a Varignon parallelogram. In such a local chart, the differentials $dz$ and $d\bar{z}$ are canonically defined. A discrete form is of type $\Diamond$ if and only if, restricted to $F_Q$, it can be represented as $\omega = p dz + q d\bar{z}$. Due to the parallelism of opposite edges, such a form naturally takes identical values on opposite edges.
\end{remark}

\begin{definition}
	The \textit{integral} of a discrete one-form $\omega$ along a directed path $P$ on the medial graph is defined as the sum of the edge values:
	\[ \int_P \omega \coloneq \sum_{e \in P} \omega(e). \]
\end{definition}

We define the exterior derivative and the concept of exact forms using the discrete Stokes' theorem.

\begin{definition}
	The \textit{discrete exterior derivative} $d$ is defined as follows:
	\begin{itemize}
		\item For a function $f: V(\Lambda) \to \mathds{C}$, its differential $df$ is a one-form on the medial graph $X$. Let $e$ be a directed edge of $X$ connecting the midpoint of an edge $(u,v)$ to the midpoint of an edge $(v,w)$ of $\Lambda$ (where $u,v,w \in V(\Lambda)$). We define $df(e)$ as the difference of the arithmetic means of the function values at the endpoints of the corresponding edges of $\Lambda$:
		\[ df(e) \coloneq \frac{f(v) + f(w)}{2} - \frac{f(u) + f(v)}{2} = \frac{f(w) - f(u)}{2}. \]
		\item For a one-form $\omega$, its derivative $d\omega$ is a 2-form on the faces of $X$ defined by the boundary integral:
		\[ d\omega(F) \coloneq \oint_{\partial F} \omega. \]
	\end{itemize}
	A one-form $\omega$ is called \textit{closed} if $d\omega = 0$ (i.e., its sum around any face $F_Q$ or $F_v$ vanishes). It is called \textit{exact} if $\omega = df$ for some function $f$.
\end{definition}

\begin{remark}
	In the analytic theory developed by Bobenko and the second author \cite{BoG17}, the exterior derivative is defined via local chart representations, mimicking the classical theory where $df = \frac{\partial f}{\partial z} dz + \frac{\partial f}{\partial \bar{z}} d\bar{z}$. In that framework, the relation $\int_{\partial F} \omega = \iint_F d\omega$ (Stokes' theorem) is a non-trivial result derived from the geometry. However, following the original combinatorial approach of Mercat \cite{Me01}, we simplify the exposition here by ignoring charts and defining the derivative directly via the Stokes property. This allows us to work consistently without explicit reference to the geometric realization.
\end{remark}

Finally, we define discrete holomorphic differentials. In the continuous setting, holomorphic differentials are closed one-forms that can locally be written as $g(z) dz$. Ideally, we would look for closed forms $\omega$ that locally satisfy $\omega = p dz$. However, to maintain our chart-independent approach, we utilize the fact that locally any such form is the differential of a holomorphic function.

\begin{definition}
	A discrete one-form $\omega$ of type $\Diamond$ is called \textit{discrete holomorphic} if:
	\begin{enumerate}
		\item It is \textit{closed} ($d\omega = 0$).
		\item It is \textit{locally exact} with respect to a discrete holomorphic function. That is, for every quadrilateral $Q \in F(\Lambda)$, there exists a discrete holomorphic function $f$ such that $\omega\vert_Q = df$.
	\end{enumerate}
\end{definition}

This definition ensures that the periods of discrete holomorphic differentials are well-defined (due to closedness) and that they carry the correct complex structure induced by $\rho$ (due to the local relationship with holomorphic functions).

\subsection{Discrete period matrices}\label{sec:period}

We consider the symplectic homology basis $\{a_1, \ldots, a_g, b_1, \ldots, b_g\}$ of $H_1(\Sigma,\Z)$ introduced in Section~\ref{sec:continuous}. Let $\alpha_1, \ldots, \alpha_g, \beta_1, \ldots, \beta_g$ be closed paths on the medial graph $X$ representing these homology classes.

As illustrated in Figure~\ref{fig:contours2}, any oriented cycle $P$ on $X$ naturally induces two parallel cycles: $B(P)$ on the black graph $\Gamma$ and $W(P)$ on the white graph $\Gamma^*$. Specifically, $B(P)$ consists of the edges of $\Gamma$ corresponding to the black edges of $P$, while $W(P)$ consists of the edges of $\Gamma^*$ corresponding to the white edges of $P$. We denote the sets of oriented edges of $X$ that are parallel to the edges of $B(P)$ and $W(P)$ as $BP$ and $WP$, respectively.

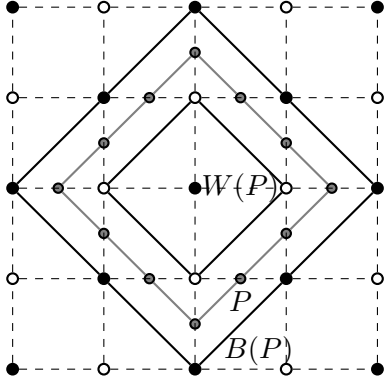
\begin{figure}[htbp]
	\centering
	\begin{tikzpicture}[
		scale=1.2,
		>=Latex,
		thick,
		vertexBlack/.style={circle,draw=black,fill=black,inner sep=0pt,minimum size=4pt},
		vertexWhite/.style={circle,draw=black,fill=white,inner sep=0pt,minimum size=4pt},
		vertexMedial/.style={circle,draw=black,fill=gray,inner sep=0pt,minimum size=3.5pt}
		]
		\node[vertexBlack] (w1) at (-2,-2) {};
		\node[vertexBlack] (w2) at (0,-2) {};
		\node[vertexBlack] (w3) at (2,-2) {};
		\node[vertexBlack] (w4) at (-1,-1) {};
		\node[vertexBlack] (w5) at (1,-1) {};
		\node[vertexBlack] (w6) at (-2,0) {};
		\node[vertexBlack] (w7) at (0,0) {};
		\node[vertexBlack] (w8) at (2,0) {};
		\node[vertexBlack] (w9) at (-1,1) {};
		\node[vertexBlack] (w10) at (1,1) {};
		\node[vertexBlack] (w11) at (-2,2) {};
		\node[vertexBlack] (w12) at (0,2) {};
		\node[vertexBlack] (w13) at (2,2) {};
		
		\node[vertexWhite] (b1) at (-1,-2) {};
		\node[vertexWhite] (b2) at (1,-2) {};
		\node[vertexWhite] (b3) at (-2,-1) {};
		\node[vertexWhite] (b4) at (0,-1) {};
		\node[vertexWhite] (b5) at (2,-1) {};
		\node[vertexWhite] (b6) at (-1,0) {};
		\node[vertexWhite] (b7) at (1,0) {};
		\node[vertexWhite] (b8) at (-2,1) {};
		\node[vertexWhite] (b9) at (0,1) {};
		\node[vertexWhite] (b10) at (2,1) {};
		\node[vertexWhite] (b11) at (-1,2) {};
		\node[vertexWhite] (b12) at (1,2) {};
		
		\node[vertexMedial] (m1) at (0,-1.5) {};
		\node[vertexMedial] (m2) at (0.5,-1) {};
		\node[vertexMedial] (m3) at (1,-0.5) {};
		\node[vertexMedial] (m4) at (1.5,0) {};
		\node[vertexMedial] (m5) at (1,0.5) {};
		\node[vertexMedial] (m6) at (0.5,1) {};
		\node[vertexMedial] (m7) at (0,1.5) {};
		\node[vertexMedial] (m8) at (-0.5,1) {};
		\node[vertexMedial] (m9) at (-1,0.5) {};
		\node[vertexMedial] (m10) at (-1.5,0) {};
		\node[vertexMedial] (m11) at (-1,-0.5) {};
		\node[vertexMedial] (m12) at (-0.5,-1) {};
		
		\draw[dashed, thin] (w1) -- (b1) -- (w2) -- (b2) -- (w3);
		\draw[dashed, thin] (b3) -- (w4) -- (b4) -- (w5) -- (b5);
		\draw[dashed, thin] (w6) -- (b6) -- (w7) -- (b7) -- (w8);
		\draw[dashed, thin] (b8) -- (w9) -- (b9) -- (w10) -- (b10);
		\draw[dashed, thin] (w11) -- (b11) -- (w12) -- (b12) -- (w13);
		
		\draw[dashed, thin] (w1) -- (b3) -- (w6) -- (b8) -- (w11);
		\draw[dashed, thin] (b1) -- (w4) -- (b6) -- (w9) -- (b11);
		\draw[dashed, thin] (w2) -- (b4) -- (w7) -- (b9) -- (w12);
		\draw[dashed, thin] (b2) -- (w5) -- (b7) -- (w10) -- (b12);
		\draw[dashed, thin] (w3) -- (b5) -- (w8) -- (b10) -- (w13);
		
		\draw (w2) -- (w5) -- (w8) -- (w10) -- (w12) -- (w9) -- (w6) -- (w4) -- (w2);
		
		\draw (b4) -- (b7) -- (b9) -- (b6) -- (b4);
		
		\draw[gray] (m1) -- (m2) -- (m3) -- (m4) -- (m5) -- (m6) -- (m7) -- (m8) -- (m9) -- (m10) -- (m11) -- (m12) -- (m1);
		
		\coordinate[label=center:$W(P)$] (z1) at (0.5,0);
		\coordinate[label=center:$P$] (z2) at (0.5,-1.25);
		\coordinate[label=center:$B(P)$] (z3) at (0.7,-1.8);
		
	\end{tikzpicture}
	\caption{A cycle $P$ on the medial graph $X$ (gray), and the induced cycles $B(P)$ on $\Gamma$ (black) and $W(P)$ on $\Gamma^*$ (white).}
	\label{fig:contours2}
\end{figure}

\begin{definition}\label{def:blackwhiteperiod}
	Let $\omega$ be a closed discrete differential of type $\Diamond$. For $1 \leq k \leq g$, we define its standard periods as
	\[ A_k \coloneq \oint_{\alpha_k} \omega \quad \text{and} \quad B_k \coloneq \oint_{\beta_k} \omega. \]
	Utilizing the bipartite structure, we define the \textit{black periods}:
	\[ A^B_k \coloneq 2\int_{B\alpha_k} \omega \quad \text{and} \quad B^B_k \coloneq 2\int_{B\beta_k} \omega. \]
	Similarly, we define the \textit{white periods}:
	\[ A^W_k \coloneq 2\int_{W\alpha_k} \omega \quad \text{and} \quad B^W_k \coloneq 2\int_{W\beta_k} \omega. \]
\end{definition}

Since $\omega$ is of type $\Diamond$, the integral along the cycle splits additively. The factor 2 ensures that the standard period is the arithmetic mean of the black and white periods:
\[ 2A_k = A_k^B + A_k^W \quad \text{and} \quad 2B_k = B_k^B + B_k^W. \]
Note that the discrete periods of closed differentials depend only on the homology classes of the paths.

It was shown in \cite{BoG17} that the dimension of the space of discrete holomorphic differentials is $2g$, twice the dimension of the smooth case. The splitting into black and white periods allows for a well-posed period problem, summarized in the following proposition \cite{BoG17}:

\begin{proposition}\label{prop:holomorphic_existence}
	For any set of $2g$ complex numbers $A_k^B, A_k^W$ ($1 \leq k \leq g$), there exists a unique discrete holomorphic differential $\omega$ with these specified black and white $a$-periods.
\end{proposition}

This existence result allows us to define the discrete period matrices.

\begin{definition}\label{def:discrete_period_matrix}
	We define a canonical basis of $2g$ discrete holomorphic differentials:
	\begin{itemize}
		\item Let $\omega_k^B$ ($1 \leq k \leq g$) be the unique differential with black $a_j$-periods $\delta_{jk}$ and vanishing white $a$-periods.
		\item Let $\omega_k^W$ ($1 \leq k \leq g$) be the unique differential with white $a_j$-periods $\delta_{jk}$ and vanishing black $a$-periods.
	\end{itemize}
	We refer to $\{\omega_1^B, \dots, \omega_g^B, \omega_1^W, \dots, \omega_g^W\}$ as the \textit{dual basis of discrete holomorphic differentials}.
	
	The $b$-periods of this basis form the blocks of the period matrix. We define the $(g \times g)$-matrices:
	\begin{align*}
		\Pi^{B,B}_{jk} &\coloneq 2\int_{B\beta_j}\omega^B_k, & \Pi^{W,B}_{jk} &\coloneq 2\int_{W\beta_j}\omega^B_k,\\
		\Pi^{B,W}_{jk} &\coloneq 2\int_{B\beta_j}\omega^W_k, & \Pi^{W,W}_{jk} &\coloneq 2\int_{W\beta_j}\omega^W_k.
	\end{align*}
	
	The \textit{complete discrete period matrix} is the $(2g\times 2g)$-matrix:
	\[ \tilde{\Pi} \coloneq \begin{pmatrix} \Pi^{B,W} & \Pi^{B,B}\\ \Pi^{W,W} & \Pi^{W,B}\end{pmatrix}. \]
	
	The \textit{discrete period matrix} $\Pi$ is defined as the arithmetic mean:
	\[ \Pi \coloneq \frac{1}{2}(\Pi^{B,W} + \Pi^{B,B} + \Pi^{W,W} + \Pi^{W,B}). \]
\end{definition}

\begin{remark}
	Alternatively, $\Pi$ can be defined directly using a set of $g$ discrete differentials characterized by equal black and white $a$-periods (i.e., $A^B_k = A^W_k = \delta_{jk}$). In this case, $\Pi_{jk} = \oint_{\beta_j} \omega_k$.
\end{remark}

The structural properties of these matrices mirrors the continuous theory, as shown in \cite{BoG17} and \cite{G23}:

\begin{proposition}\label{prop:period_matrix}
	The matrices $\Pi$ and $\tilde{\Pi}$ are symmetric, and their imaginary parts are positive definite. Furthermore, the imaginary parts of the off-diagonal blocks $\Pi^{W,B}$ and $\Pi^{B,W}$ (in the complete matrix) are also positive definite.
\end{proposition}

For the special class of orthodiagonal surfaces (where $\rho_Q \in \mathds{R}^+$), we observe a specific decoupling of real and imaginary parts.

\begin{lemma}\label{lem:period_orthodiagonal}
	Let $(\Sigma, \Lambda)$ be an orthodiagonal discrete Riemann surface. Then, the matrices $\Pi^{W,B}$ and $\Pi^{B,W}$ are purely imaginary, while the matrices $\Pi^{B,B}$ and $\Pi^{W,W}$ are real.
\end{lemma}

In particular, the discrete period matrix $\Pi$ lies in the Siegel upper half-space. In the context of finer discretizations of a fixed polyhedral surface (viewed as a Riemann surface), it has been shown in \cite{G23} that the discrete period matrices converge to their continuous counterparts.
In the following chapter, we will extend this theory to discrete real Riemann surfaces, showing that the discrete period matrices inherit the symmetry properties of the continuous case.


\section{Discrete real Riemann surfaces}\label{sec:discrete_real}

The aim of this chapter is to define discrete real Riemann surfaces and develop the discrete theory in strict analogy to the continuous theory presented in Chapter~\ref{sec:continuous}.
Just as a classical real Riemann surface is equipped with an antiholomorphic involution reversing the orientation, we will define a combinatorial involution on the quad-graph that acts on the discrete complex structure in a compatible way.

We begin by defining discrete antiholomorphic involutions in Section~\ref{sec:discrete_involution}, distinguishing between two fundamental types based on the bipartite structure. In Section~\ref{sec:discrete_ovals}, we introduce discrete real ovals and relate their number to the topological type. The core of the chapter, Section~\ref{sec:discrete_homology}, is dedicated to the construction of a symplectic homology basis adapted to the involution, which allows us to derive the symmetry properties of the discrete period matrix. We conclude with illustrative examples in Section~\ref{sec:examples}.

\subsection{Discrete antiholomorphic involution}\label{sec:discrete_involution}

Let $(\Sigma, \Lambda)$ be a discrete Riemann surface with discrete complex structure $\rho$. A discrete antiholomorphic involution $\tau$ is, first and foremost, a combinatorial automorphism of the cell decomposition that reverses orientation.

Since $\tau$ is an involution on the surface $\Sigma$, it must map the cellular decomposition $\Lambda$ to itself. Specifically, $\tau$ acts bijectively on the set of vertices $V(\Lambda)$, mapping edges to edges and quadrilaterals to quadrilaterals.
The orientation-reversing property is encoded locally: if vertices $v_1, v_2, v_3, v_4$ form the boundary of a face $Q$ in counterclockwise order, their images $\tau(v_1), \tau(v_2), \tau(v_3), \tau(v_4)$ must form the boundary of the image face $\tau(Q)$ in \emph{clockwise} order.

Crucially, the bipartiteness of $\Lambda$ implies that such a graph automorphism must either preserve the color classes or swap them. This leads to two distinct topological types of involutions (see Figure~\ref{fig:definition_tau}):
\begin{enumerate}
	\item \textbf{Type 1 (color-preserving):} $\tau$ maps black vertices to black vertices and white vertices to white vertices. Geometrically, this corresponds to a reflection along an edge or across the face of the quad-graph.
	\item \textbf{Type 2 (color-reversing):} $\tau$ swaps black and white vertices. Geometrically, this corresponds to a reflection along a diagonal of a face of the medial graph (connecting the midpoints of two opposite edges of the quadrilateral).
\end{enumerate}

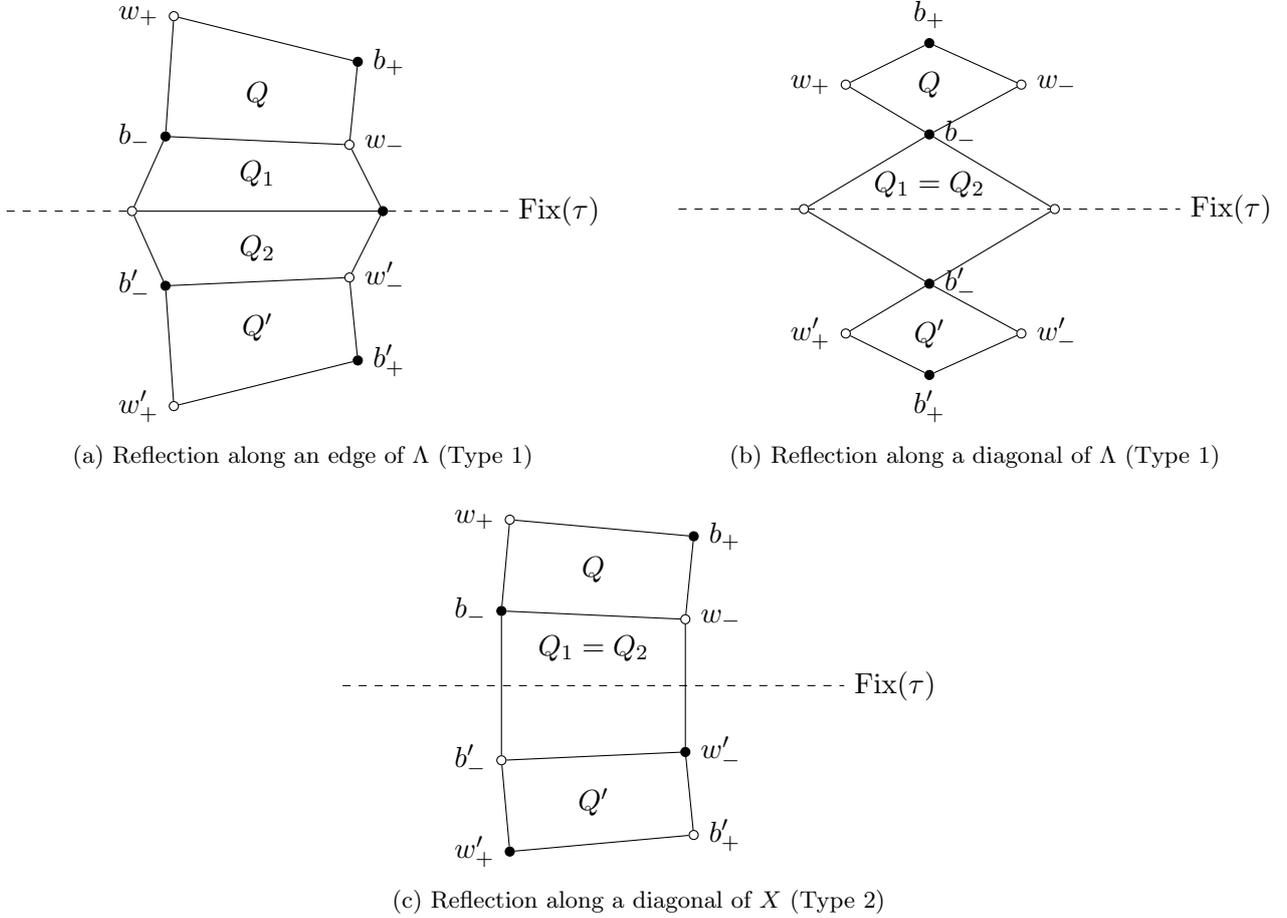
\begin{figure}[htbp]
	\centering
	\subfloat[Reflection along an edge of $\Lambda$ (Type 1)]{
	\begin{tikzpicture}[white/.style={circle,draw=black,fill=white,thin,inner sep=0pt,minimum size=1.2mm},
		black/.style={circle,draw=black,fill=black,thin,inner sep=0pt,minimum size=1.2mm},
		gray/.style={circle,draw=black,fill=gray,thin,inner sep=0pt,minimum size=1.2mm},scale=1.1]
		
		\node[white] (wplus) [label=left:$w_+$] at (-1,2.35) {};
		\node[white] (wminus) [label=right:$w_-$] at (1.1,0.8) {};
		\node[black] (bplus) [label=right:$b_+$] at (1.2,1.8) {};
		\node[black] (bminus) [label=left:$b_-$] at (-1.1,0.9) {};
		
		\draw (wplus) -- (bplus) -- (wminus) -- (bminus) -- (wplus);
		\node at (0,1.4) {$Q$};
		
		\node[white] (wplusp) [label=left:$w^\prime_+$] at (-1,-2.35) {};
		\node[white] (wminusp) [label=right:$w^\prime_-$] at (1.1,-0.8) {};
		\node[black] (bplusp) [label=right:$b^\prime_+$] at (1.2,-1.8) {};
		\node[black] (bminusp) [label=left:$b^\prime_-$] at (-1.1,-0.9) {};
		
		\draw (wplusp) -- (bplusp) -- (wminusp) -- (bminusp) -- (wplusp);
		\node at (0,-1.4) {$Q^\prime$};
		
		\node[black] (b) at (1.5,0) {};
		\node[white] (w) at (-1.5,0) {};
		\draw (b) -- (w);
		\draw (bminus) -- (w) -- (bminusp);
		\draw (wminus) -- (b) -- (wminusp);
		
		\draw[dashed] (-3,0) -- (w);
		\draw[dashed] (b) -- (3,0) node[right] {$\fix(\tau)$};
		
		\node at (0,0.45) {$Q_1$};
		\node at (0,-0.45) {$Q_2$};
	\end{tikzpicture}}
	\qquad
	\subfloat[Reflection along a diagonal of $\Lambda$ (Type 1)]{
	\begin{tikzpicture}[white/.style={circle,draw=black,fill=white,thin,inner sep=0pt,minimum size=1.2mm},
		black/.style={circle,draw=black,fill=black,thin,inner sep=0pt,minimum size=1.2mm},
		gray/.style={circle,draw=black,fill=gray,thin,inner sep=0pt,minimum size=1.2mm},scale=1.1]
		\draw[dashed] (-3,0) -- (3,0) node[right] {$\fix(\tau)$};
		
		\node[white] (wplus) [label=left:$w_+$] at (-1,1.5) {};
		\node[white] (wminus) [label=right:$w_-$] at (1.1,1.5) {};
		\node[black] (bplus) [label=above:$b_+$] at (0,2) {};
		\node[black] (bminus) [label=right:$b_-$] at (0,0.9) {};
		
		\draw (wplus) -- (bplus) -- (wminus) -- (bminus) -- (wplus);
		\node at (0,1.5) {$Q$};
		
		\node[white] (wplusp) [label=left:$w^\prime_+$] at (-1,-1.5) {};
		\node[white] (wminusp) [label=right:$w^\prime_-$] at (1.1,-1.5) {};
		\node[black] (bplusp) [label=below:$b^\prime_+$] at (0,-2) {};
		\node[black] (bminusp) [label=right:$b^\prime_-$] at (0,-0.9) {};
		
		\draw (wplusp) -- (bplusp) -- (wminusp) -- (bminusp) -- (wplusp);
		\node at (0,-1.5) {$Q^\prime$};
		
		\node[white] (wp) at (1.5,0) {};
		\node[white] (w) at (-1.5,0) {};
		\draw (bminus) -- (wp) -- (bminusp) -- (w) -- (bminus);
		
		\node at (0,0.3) {$Q_1=Q_2$};
	\end{tikzpicture}}
	\qquad
	\subfloat[Reflection along a diagonal of $X$ (Type 2)]{
	\begin{tikzpicture}[white/.style={circle,draw=black,fill=white,thin,inner sep=0pt,minimum size=1.2mm},
		black/.style={circle,draw=black,fill=black,thin,inner sep=0pt,minimum size=1.2mm},
		gray/.style={circle,draw=black,fill=gray,thin,inner sep=0pt,minimum size=1.2mm},scale=1.1]
		\draw[dashed] (-3,0) -- (3,0) node[right] {$\fix(\tau)$};
		
		\node[white] (wplus) [label=left:$w_+$] at (-1,2) {};
		\node[white] (wminus) [label=right:$w_-$] at (1.1,0.8) {};
		\node[black] (bplus) [label=right:$b_+$] at (1.2,1.8) {};
		\node[black] (bminus) [label=left:$b_-$] at (-1.1,0.9) {};
		
		\draw (wplus) -- (bplus) -- (wminus) -- (bminus) -- (wplus);
		\node at (0,1.4) {$Q$};
		
		\node[black] (wplusp) [label=left:$w^\prime_+$] at (-1,-2) {};
		\node[black] (wminusp) [label=right:$w^\prime_-$] at (1.1,-0.8) {};
		\node[white] (bplusp) [label=right:$b^\prime_+$] at (1.2,-1.8) {};
		\node[white] (bminusp) [label=left:$b^\prime_-$] at (-1.1,-0.9) {};
		
		\draw (wplusp) -- (bplusp) -- (wminusp) -- (bminusp) -- (wplusp);
		\node at (0,-1.4) {$Q^\prime$};
		
		\draw (bminus) -- (bminusp);
		\draw (wminus) -- (wminusp);
		
		\node at (0,0.45) {$Q_1=Q_2$};
	\end{tikzpicture}}
	\caption[]{Local behavior of discrete antiholomorphic involutions $\tau$. (a) and (b) preserve the bipartition, while (c) swaps black and white vertices.}
	\label{fig:definition_tau}
\end{figure}

To qualify as \emph{antiholomorphic}, $\tau$ must interact with the discrete complex structure $\rho$ in a way that corresponds to complex conjugation in local charts.
\begin{itemize}
	\item \textbf{Motivation for Type 1:} Since $\tau$ preserves the coloring, it maps the black diagonal of $Q$ to the black diagonal of $\tau(Q)$, and similarly for white diagonals. Realizing $\tau$ locally as complex conjugation ($z \mapsto \bar{z}$) yields the condition $\rho_{\tau(Q)} = \overline{\rho}_Q$.
	\item \textbf{Motivation for Type 2:} Since $\tau$ swaps the coloring, it maps the black diagonal of $Q$ to the \emph{white} diagonal of $\tau(Q)$ and vice versa. This exchange inverts the ratio defining the discrete complex structure, leading to the condition $\rho_{\tau(Q)} = 1/\overline{\rho}_Q$.
\end{itemize}

Based on this motivation, we define:

\begin{definition} \label{def:discrete_real_RS}
	A map $\tau : V(\Lambda) \to V(\Lambda)$ is called a \emph{discrete antiholomorphic involution} if it is a graph automorphism satisfying $\tau \circ \tau = \mathrm{id}$, it maps faces to faces (i.e., for every $Q \in F(\Lambda)$, the image $\tau(Q)$ is a quadrilateral in $F(\Lambda)$), and it satisfies one of the following two sets of conditions:
	\begin{enumerate}
		\item[(i)] \textbf{Color-preserving:}
		\begin{itemize}
			\item $\tau(\Gamma) = \Gamma$ and $\tau(\Gamma^*) = \Gamma^*$;
			\item For every $Q \in F(\Lambda)$, $\rho_{\tau(Q)} = \overline{\rho}_{Q}$.
		\end{itemize}
		\item[(ii)] \textbf{Color-reversing:}
		\begin{itemize}
			\item $\tau(\Gamma) = \Gamma^*$ and $\tau(\Gamma^*) = \Gamma$;
			\item For every $Q \in F(\Lambda)$, $\rho_{\tau(Q)} = \frac{1}{\overline{\rho}_{Q}}$.
		\end{itemize}
	\end{enumerate}
	The triple $(\Sigma, \Lambda, \tau)$ is called a \emph{discrete real Riemann surface}.
\end{definition}

\begin{remark}
	The involution $\tau$ naturally extends to the edges and faces of $\Lambda$. Consequently, it induces an involution on the medial graph $X$, mapping vertices $V(X)$ (midpoints of edges of $\Lambda$) to vertices $V(X)$, and edges $E(X)$ to edges $E(X)$.
	
	The distinction between the two types is crucial for the medial graph:
	\begin{itemize}
		\item In Case (i), $\tau$ maps black edges of $X$ to black edges, and white edges to white edges.
		\item In Case (ii), $\tau$ maps black edges of $X$ to white edges, and vice versa.
	\end{itemize}
	This swapping of edge types in Case (ii) will play a significant role when analyzing the action of $\tau$ on the period matrix.
\end{remark}

\subsection{Discrete real ovals}\label{sec:discrete_ovals}

From now on, let $(\Sigma, \Lambda, \tau)$ be a discrete real Riemann surface of genus $g$. We aim to characterize the set $\fix(\tau)$ of fixed points of $\tau$.

Recall that $\tau$ acts on the cellular decomposition $\Lambda$ and induces an involution on the medial graph $X$. The fixed point set depends on whether $\tau$ preserves or reverses the bipartite coloring (Type 1 or Type 2).

\subsubsection*{Local structure of the fixed point set}

\paragraph{Case 1: color-preserving (Type 1).}
If $\tau$ preserves the coloring, it maps black vertices to black and white to white. A fixed point can be:
\begin{itemize}
	\item A vertex $v \in V(\Lambda)$ such that $\tau(v) = v$.
	\item An edge $e \in E(\Lambda)$ such that $\tau(e) = e$ (which implies both endpoints are fixed).
	\item A face $Q \in F(\Lambda)$ such that $\tau(Q) = Q$. Since $\tau$ reverses orientation and preserves colors, it must reflect $Q$ across a diagonal. The fixed points within $Q$ correspond to this diagonal (either connecting two black or two white vertices).
\end{itemize}
Thus, $\fix(\tau)$ consists of vertices, edges, and diagonals of $\Lambda$. Geometrically, this corresponds to reflection along a one-dimensional subgraph embedded in the set of edges and diagonals of $\Lambda$ (see Figure~\ref{fig:definition_tau} (a,b)).

\paragraph{Case 2: color-reversing (Type 2).}
If $\tau$ swaps colors, it cannot fix any vertex or edge of $\Lambda$. However, it may fix elements of the medial graph $X$:
\begin{itemize}
	\item A vertex $x \in V(X)$ (midpoint of an edge of $\Lambda$) is fixed if $\tau$ swaps the endpoints of the corresponding edge of $\Lambda$ and maps the edge to itself.
	\item A face $Q \in F(\Lambda)$ is fixed ($\tau(Q)=Q$) if $\tau$ maps the quadrilateral to itself while swapping colors. Since $\tau$ reverses orientation, it must reflect $Q$ across a bimedian (a line connecting the midpoints of opposite edges). In terms of the medial graph, this corresponds to a diagonal of the face $F_Q$ of $X$.
\end{itemize}
Thus, $\fix(\tau)$ consists of vertices and diagonals of the medial graph $X$ (see Figure~\ref{fig:definition_tau} (c)).

We summarize these observations:

\begin{definition}
	Let $(\Sigma, \Lambda, \tau)$ be a discrete real Riemann surface. The \textit{set of fixed points} $\fix(\tau)$ is defined as the union of the following sets:
	\begin{itemize}
		\item If $\tau$ is of Type 1 (color-preserving): All vertices $v \in V(\Lambda)$ fixed by $\tau$, all edges $e \in E(\Lambda)$ fixed by $\tau$, and all diagonals of faces $Q \in F(\Lambda)$ fixed by $\tau$.
		\item If $\tau$ is of Type 2 (color-reversing): All vertices $x \in V(X)$ fixed by $\tau$ (viewed as midpoints of edges of $\Lambda$), and all diagonals of faces $F_Q$ of the medial graph $X$ corresponding to faces $Q \in F(\Lambda)$ fixed by $\tau$.
	\end{itemize}
\end{definition}

\begin{proposition}\label{prop:fixedpoint2}
	If $\fix(\tau)$ is non-empty, it forms a collection of disjoint closed curves embedded in the surface. These curves lie either on the set of edges and diagonals of $\Lambda$ or on the diagonals of $X$ corresponding to faces of $\Lambda$.
\end{proposition}

\begin{definition}
	The disjoint closed curves of $\fix(\tau)$ are called the \emph{discrete real ovals} of $(\Sigma, \Lambda, \tau)$. A discrete real Riemann surface of genus $g$ with $g+1$ discrete real ovals is called a \emph{discrete M-curve}.
\end{definition}

\subsubsection*{Topological properties}

To transfer the topological results (such as Harnack's inequality~\ref{prop:number_ovals}) to the discrete setting, we construct a continuous model.
We realize $(\Sigma, \Lambda)$ as a polyhedral surface $S$ composed of unit squares, glued according to the combinatorics of $\Lambda$. The involution $\tau$ on the vertices $V(\Lambda)$ induces a map on the corners of these squares. We extend this map to a continuous, orientation-reversing involution $\tilde{\tau}: S \to S$ by linear interpolation on each square.


Specifically, for a square $Q$ with vertices $v_1, \dots, v_4$, we identify $Q$ with $[0,1]^2$. If $\tau$ maps $Q$ to a square $Q'$, the map $\tilde{\tau}$ is the unique affine map from $Q$ to $Q'$ consistent with the action on vertices.

\begin{lemma}\label{lem:fixedtau}
	The fixed point set of the continuous involution $\tilde{\tau}$ coincides exactly with the geometric realization of the discrete fixed point set $\fix(\tau)$.
\end{lemma}

\begin{proof}
	By construction, fixed points of $\tilde{\tau}$ on the boundary of squares correspond to fixed vertices or edges (or midpoints of edges in Type 2). Inside a square $Q$, a fixed point exists if and only if $\tau(Q)=Q$.
	\begin{itemize}
		\item In Type 1, $\tau$ acts on the vertices of $Q$ as a reflection (interchanging $v_2 \leftrightarrow v_4$, fixing $v_1, v_3$). The linear interpolation $\tilde{\tau}$ is then the geometric reflection across the diagonal $v_1v_3$.
		\item In Type 2, $\tau$ acts by swapping colors ($v_1 \leftrightarrow v_2, v_3 \leftrightarrow v_4$). The linear interpolation $\tilde{\tau}$ is the reflection across the bimedian connecting the midpoints of $v_1v_2$ and $v_3v_4$.
	\end{itemize}
	In both cases, the fixed point set of $\tilde{\tau}$ is exactly the geometric locus described in the definition of $\fix(\tau)$.
\end{proof}

Since $\Sigma$ and $S$ are homeomorphic and $\tilde{\tau}$ is a continuous orientation-reversing involution, we can directly apply the topological results from the smooth theory (Section~\ref{sec:continuous}) to our discrete setting.

\begin{lemma}\label{lem:12components_discrete}
	Let $(\Sigma, \Lambda, \tau)$ be a discrete real Riemann surface. Then $\Sigma \setminus \fix(\tau)$ consists of either one or two connected components.
\end{lemma}

\begin{definition}
	A discrete real Riemann surface $(\Sigma, \Lambda, \tau)$ is called
	\begin{itemize}
		\item \emph{dividing} if $\Sigma \setminus \fix(\tau)$ consists of two components (in this case $S / \tilde{\tau}$ is orientable).
		\item \emph{non-dividing} if $\Sigma \setminus \fix(\tau)$ is connected (in this case $S / \tilde{\tau}$ is non-orientable).
	\end{itemize}
\end{definition}

The topological classification of discrete real Riemann surfaces follows directly from the properties of the continuous model constructed above. Since the fixed point set structure and the orientability of the quotient are topological invariants preserved by the homeomorphism between $(\Sigma, \fix(\tau))$ and $(S, \fix(\tilde{\tau}))$, we obtain the discrete analogue of Proposition~\ref{prop:number_ovals}.

\begin{theorem}\label{th:number_ovals_discrete}
	Let $(\Sigma, \Lambda, \tau)$ be a discrete real Riemann surface of genus $g$ and let $k$ be the number of discrete real ovals. Then, the following holds:
	\begin{enumerate}
		\item[(i)] \textbf{Harnack's Inequality:} The number of ovals is bounded by the genus:
		\[ 0 \leq k \leq g+1. \]
		\item[(ii)] If $k=g+1$ (discrete M-curve), then $(\Sigma, \Lambda, \tau)$ is dividing. If $k=0$, then $(\Sigma, \Lambda, \tau)$ is non-dividing.
		\item[(iii)] If $(\Sigma, \Lambda, \tau)$ is dividing, then the number of ovals satisfies the congruence:
		\[ k \equiv g+1 \pmod 2. \]
	\end{enumerate}
\end{theorem}

\subsection{Discrete homology basis and discrete period matrices}\label{sec:discrete_homology}

The topological classification of discrete real Riemann surfaces allows us to construct a homology basis adapted to the involution $\tau$. By transferring the results from the continuous theory via the polyhedral model, we obtain the discrete analogue of Proposition~\ref{prop:realhomology}.

\begin{theorem}\label{th:realhomology_discrete}
	Let $(\Sigma, \Lambda, \tau)$ be a discrete real Riemann surface of genus $g$ with $k$ discrete real ovals. Then, there exists a symplectic basis $(a_1, \ldots, a_g, b_1, \ldots, b_g)$ of $H_1(X,\Z)$ such that the induced action of $\tau$ on homology is given by:
	\begin{align*}
		\tau(a_i) &= a_i \quad \text{for } 1 \leq i \leq g,\\
		\tau(b_i) &= \sum_{j=1}^g h_{ji} a_j - b_i \quad \text{for } 1 \leq i \leq g,
	\end{align*}
	where $H = (h_{ji})$ is a symmetric matrix with entries in $\{0,1\}$. Furthermore:
	\begin{enumerate}
		\item[(i)] If $(\Sigma, \Lambda, \tau)$ is dividing, then $\mathrm{diag}(H) = 0$ and $\mathrm{rank}(H) = g + 1 - k$.
		\item[(ii)] If $(\Sigma, \Lambda, \tau)$ is non-dividing and $k \neq 0$, then $\mathrm{diag}(H) = 1$ and $\mathrm{rank}(H) = g + 1 - k$.
		\item[(iii)] If $(\Sigma, \Lambda, \tau)$ is non-dividing and $k = 0$, then $\mathrm{diag}(H) = 0$ and
		\[
		\mathrm{rank}(H) =
		\begin{cases}
			g & \text{if } g \text{ is even}, \\
			g - 1 & \text{if } g \text{ is odd}.
		\end{cases}
		\]
	\end{enumerate}
\end{theorem}

\begin{proof}
	As established in Section~\ref{sec:discrete_ovals}, the discrete real Riemann surface $(\Sigma, \Lambda, \tau)$ is homeomorphic to the polyhedral surface $(S, \tilde{\tau})$. The medial graph $X$ embeds naturally into $S$ as a deformation retract, implying $H_1(X, \Z) \cong H_1(S, \Z)$. Moreover, the combinatorial involution $\tau$ on $X$ induces the same map on homology as the continuous involution $\tilde{\tau}$ on $S$. Proposition~\ref{prop:realhomology} relies purely on topological invariants, namely the orientability of the quotient and the number of fixed components. Thus, we can directly apply it to $(S, \tilde{\tau})$ to obtain the desired basis for $H_1(X, \Z)$ and the properties of the matrix $H$.
\end{proof}

To relate the action of $\tau$ on homology to the period matrix, we analyze the pullback of discrete differential forms.

\begin{definition}
	Let $\omega$ be a discrete one-form of type $\Diamond$. The \textit{pullback} $\tau^*\omega$ is the discrete one-form defined on an oriented edge $e$ of the medial graph $X$ by:
	\[ (\tau^*\omega)(e) \coloneq \omega(\tau(e)). \]
\end{definition}

Specifically, we have the following integral relation for any path $\gamma$ on $X$:
\begin{equation}\label{eq:integral}
	\int_{\gamma} \tau^*\omega = \int_{\tau(\gamma)} \omega.
\end{equation}

We now derive the symmetry properties of the discrete period matrices.
Recall that $\tilde{\Pi}$ denotes the \emph{complete discrete period matrix} containing all bipartite periods, whereas $\Pi$ is the \emph{discrete period matrix} obtained by averaging, which corresponds to the classical object.
The result depends crucially on whether $\tau$ preserves or reverses the bipartite coloring of the quad-graph.

\begin{theorem} \label{th:completeperiodmatrix}
	Let $(\Sigma, \Lambda, \tau)$ be a discrete real Riemann surface. Let $\tilde{\Pi}$ be its complete discrete period matrix with respect to the homology basis from Theorem~\ref{th:realhomology_discrete}, and let $H$ be the matrix defined therein.
	\begin{enumerate}
		\item[(i)] \textbf{Type 1 (color-preserving):} If $\tau$ preserves the bipartite coloring (i.e., $\rho_{\tau(Q)} = \overline{\rho}_{Q}$), then:
		\[ 2 \Re(\Pi^{B,B}) = H = 2 \Re(\Pi^{W,W}) \quad \text{and} \quad \Re(\Pi^{W,B}) = 0 = \Re(\Pi^{B,W}). \]
		Consequently, the complete period matrix takes the form:
		\[ \tilde{\Pi} = \begin{pmatrix}
			i \Pi_1 & i \Pi_2 + \frac{1}{2} H \\
			i \Pi_2^T + \frac{1}{2} H & i \Pi_3
		\end{pmatrix}, \]
		where $\Pi_1, \Pi_2, \Pi_3 \in \R^{g \times g}$ are real matrices.
		\item[(ii)] \textbf{Type 2 (color-reversing):} If $\tau$ reverses the bipartite coloring (i.e., $\rho_{\tau(Q)} = 1/\overline{\rho}_{Q}$), then:
		\[ \Pi^{B,B} + \overline{\Pi}^{W,W} = H \quad \text{and} \quad \Pi^{W,B} + \overline{\Pi}^{B,W} = 0. \]
	\end{enumerate}
\end{theorem}

\begin{proof}
	The proof proceeds in analogy to the proof of Proposition~\ref{prop:matrix_smooth}, utilizing the duality of the basis. Let $\{\omega_k^B,\omega_k^W \mid 1 \leq k \leq g\}$ be the basis of discrete holomorphic differentials dual to the homology basis given in Theorem~\ref{th:realhomology_discrete}. Let $h_{ij}$ denote the entries of $H$.
	
	(i) \textbf{Type 1 (color-preserving):} Assume that $\tau$ preserves the bipartite coloring. We define the discrete one-form $\eta_j^B \coloneq \overline{\tau^* \omega_j^B}$.
	
	First, we verify that $\eta_j^B$ is discrete holomorphic. Closedness follows immediately from the closedness of $\omega_j^B$ and the linearity of the pullback and conjugation. For the second condition, we must show that $\eta_j^B$ is locally exact with respect to a discrete holomorphic function.
	Let $Q \in F(\Lambda)$ be a quadrilateral. Since $\omega_j^B$ is discrete holomorphic, there exists a discrete holomorphic function $f$ on $\tau(Q)$ such that $\omega_j^B|_{\tau(Q)} = df$. We define the function $g$ on the vertices of $Q$ by $g(v) \coloneqq \overline{f(\tau(v))}$. Then, for any edge $e=(u,v)$ of the medial graph inside $Q$ (connecting edges of $\Lambda$), we have:
	\[
	\eta_j^B(e) = \overline{\omega_j^B(\tau(e))} = \overline{df(\tau(e))} = \overline{f(\tau(v)) - f(\tau(u))} = g(v) - g(u) = dg(e).
	\]
	Thus, $\eta_j^B|_Q = dg$. It remains to show that $g$ satisfies the discrete Cauchy-Riemann equations on $Q$. Let $b_-, w_-, b_+, w_+$ be the vertices of $Q$ in counterclockwise order. Since $\tau$ reverses orientation, the vertices $\tau(b_-), \tau(w_-), \tau(b_+), \tau(w_+)$ of $\tau(Q)$ are ordered clockwise. The discrete holomorphic function $f$ satisfies the Cauchy-Riemann equation on $\tau(Q)$ with respect to the standard counterclockwise ordering. Traversing $\tau(Q)$ counterclockwise effectively swaps the relative positions of the white vertices, implying:
	\[
	f(\tau(w_-)) - f(\tau(w_+)) = i \rho_{\tau(Q)} (f(\tau(b_+)) - f(\tau(b_-))).
	\]
	Rearranging and taking the complex conjugate yields:
	\[
	\overline{f(\tau(w_+)) - f(\tau(w_-))} = \overline{- i \rho_{\tau(Q)} (f(\tau(b_+)) - f(\tau(b_-)))}.
	\]
	Using the Type 1 condition $\overline{\rho_{\tau(Q)}} = \rho_Q$, we obtain:
	\[
	g(w_+) - g(w_-) = i \rho_Q (g(b_+) - g(b_-)).
	\]
	Thus, $g$ is discrete holomorphic on $Q$, proving that $\eta_j^B$ is a discrete holomorphic differential.
	
	Since $\tau(a_i)=a_i$, the periods of $\eta_j^B$ are:
	\begin{align*}
		2\int_{Ba_i} \eta_j^B &= \overline{2\int_{\tau(Ba_i)} \omega_j^B} = \overline{2\int_{Ba_i} \omega_j^B} = \overline{\delta_{ij}} = \delta_{ij}, \\
		2\int_{Wa_i} \eta_j^B &= \overline{2\int_{\tau(Wa_i)} \omega_j^B} = \overline{2\int_{Wa_i} \omega_j^B} = 0.
	\end{align*}
	By the uniqueness of the canonical basis (Proposition~\ref{prop:holomorphic_existence}), we obtain $\eta_j^B = \omega_j^B$.
	Using the relation $\tau(b_i) = \sum h_{ki} a_k - b_i$, we calculate the $b$-periods:
	\begin{align*}
		\Pi_{ij}^{B,B} &= 2 \int_{Bb_i} \omega_j^B = 2 \int_{Bb_i} \overline{\tau^*\omega_j^B} = \overline{2 \int_{\tau(Bb_i)} \omega_j^B} \\
		&= \overline{2 \sum_{k=1}^{g} h_{ki} \int_{Ba_k} \omega_j^B - 2 \int_{Bb_i} \omega_j^B}
		= \sum_{k=1}^g h_{ki} \delta_{kj} - \overline{\Pi}_{ij}^{B,B} = h_{ij} - \overline{\Pi}_{ij}^{B,B}.
	\end{align*}
	For the off-diagonal block $\Pi^{W,B}$, recalling that the white $a$-periods of $\omega_j^B$ vanish, we get:
	\begin{align*}
		\Pi_{ij}^{W,B} &= 2 \int_{Wb_i} \omega_j^B = \overline{2 \int_{\tau(Wb_i)} \omega_j^B}
		= \overline{2 \sum_{k=1}^{g} h_{ki} \underbrace{\int_{Wa_k} \omega_j^B}_{0} - 2 \int_{Wb_i} \omega_j^B}
		= - \overline{\Pi}_{ij}^{W,B}.
	\end{align*}
	Using $h_{ji}=h_{ij}$, we conclude $2 \Re(\Pi^{B,B}) = H$ and $\Re(\Pi^{W,B}) = 0$.
	Analogously, defining $\eta_j^W \coloneqq \overline{\tau^*\omega_j^W}$ leads to $\eta_j^W = \omega_j^W$, yielding $2 \Re(\Pi^{W,W}) = H$ and $\Re(\Pi^{B,W}) = 0$.
	
	(ii) \textbf{Type 2 (color-reversing):} Assume that $\tau$ does not preserve the bipartite coloring. We define $\eta_j^B \coloneqq \overline{\tau^*\omega_j^W}$.
	
	Using the same logic as in (i), we establish that $\eta_j^B$ is discrete holomorphic. Locally, if $\omega_j^W|_{\tau(Q)} = df$, then $\eta_j^B|_Q = dg$ with $g = \overline{f \circ \tau}$. The check for the Cauchy-Riemann equations differs slightly: $\tau$ swaps black and white vertices, so the black diagonal of $Q$ maps to the white diagonal of $\tau(Q)$. Combined with the orientation reversal and the Type 2 condition $\rho_{\tau(Q)} = 1/\overline{\rho}_Q$, the Cauchy-Riemann relation for $f$ transforms into the correct relation for $g$. This transformation effectively swaps the role of the diagonals and inverts the coefficient.
	
	The period calculation yields:
	\begin{align*}
		2\int_{Ba_i} \eta_j^B &= \overline{2\int_{\tau(Ba_i)} \omega_j^W} = \overline{2\int_{Wa_i} \omega_j^W} = \overline{\delta_{ij}} = \delta_{ij},\\
		2\int_{Wa_i} \eta_j^B &= \overline{2\int_{\tau(Wa_i)} \omega_j^W} = \overline{2\int_{Ba_i} \omega_j^W} = 0.
	\end{align*}
	Thus $\eta_j^B = \omega_j^B$, meaning $\omega_j^B = \overline{\tau^* \omega_j^W}$.
	Calculating the $b$-periods:
	\begin{align*}
		\Pi_{ij}^{B,B} &= 2 \int_{Bb_i} \omega_j^B = \overline{2 \int_{\tau(Bb_i)} \omega_j^W}
		= \overline{2 \sum_{k=1}^{g} h_{ki}\int_{Wa_k}\omega_j^W - 2 \int_{Wb_i} \omega_j^W}
		= h_{ij} - \overline{\Pi}_{ij}^{W,W}.
	\end{align*}
	Similarly, $\Pi_{ij}^{W,B} = -\overline{\Pi}_{ij}^{B,W}$.
	This implies $\Pi^{B,B} + \overline{\Pi}^{W,W} = H$ and $\Pi^{W,B} + \overline{\Pi}^{B,W} = 0$.
\end{proof}

Despite the structural differences in the complete period matrix $\tilde{\Pi}$, the \emph{discrete period matrix} $\Pi$ exhibits the exact same symmetry as in the continuous case.

\begin{corollary}\label{cor:periodmatrix}
	For any discrete real Riemann surface, the discrete period matrix decomposes as
	\[ \Pi = \frac{1}{2} H + i T, \]
	where $H$ is the topological matrix from Theorem~\ref{th:realhomology_discrete} and $T$ is a real symmetric matrix.
\end{corollary}

\begin{proof}
	By definition, $2\Pi = \Pi^{B,B} + \Pi^{W,W} + \Pi^{B,W} + \Pi^{W,B}$.
	Summing the relations from Theorem~\ref{th:completeperiodmatrix}:
	\begin{itemize}
		\item Case (i): $2\Re(2\Pi) = H + H + 0 + 0 = 2H$.
		\item Case (ii): $2\Re(2\Pi) = \Re(\Pi^{B,B} + \overline{\Pi}^{W,W} + \Pi^{B,W} + \overline{\Pi}^{W,B}) = \Re(H + 0) = H$ (since $H$ is real).
	\end{itemize}
	In both cases, $\Re(\Pi) = \frac{1}{2}H$.
\end{proof}

Finally, we consider the special case of orthodiagonal surfaces (where $\rho$ is real). Recall from Lemma~\ref{lem:period_orthodiagonal} that in this case, the blocks $\Pi^{W,B}$ and $\Pi^{B,W}$ are purely imaginary, while $\Pi^{B,B}$ and $\Pi^{W,W}$ are real.
Combining this with Theorem~\ref{th:completeperiodmatrix}, we obtain the general structure for orthodiagonal surfaces.

\begin{corollary}\label{cor:matrix_real_structure}
	Let $(\Sigma, \Lambda, \tau)$ be an orthodiagonal discrete real Riemann surface.
	\begin{enumerate}
		\item[(i)] If $\tau$ preserves the bipartite coloring, then the complete discrete period matrix takes the form
		\[ \tilde{\Pi} = \begin{pmatrix}
			i \Pi_1 & \frac{1}{2} H \\
			\frac{1}{2} H & i \Pi_3
		\end{pmatrix}, \]
		where $\Pi_1, \Pi_3 \in \R^{g \times g}$ are real matrices and $H$ is the topological matrix from Theorem~\ref{th:realhomology_discrete}.
		\item[(ii)] If $\tau$ reverses the bipartite coloring, then the real blocks satisfy the relation
		\[ \Pi^{B,B} + \Pi^{W,W} = H. \]
		Furthermore, the imaginary blocks satisfy $\Pi^{W,B} = \Pi^{B,W}$.
	\end{enumerate}
\end{corollary}

For discrete M-curves, where the topological matrix vanishes ($H=0$), this structure simplifies further, revealing a skew-symmetry in the second case.

\begin{corollary}\label{cor:matrix_real_structureM}
	Let $(\Sigma, \Lambda, \tau)$ be an orthodiagonal discrete M-curve (so $H=0$).
	\begin{enumerate}
		\item[(i)] If $\tau$ preserves the bipartite coloring, then $\Pi^{B,B} = 0 = \Pi^{W,W}$. Consequently, the complete discrete period matrix $\tilde{\Pi}$ is purely imaginary.
		\item[(ii)] If $\tau$ reverses the bipartite coloring, then
		\[ \Pi^{B,B} = -\Pi^{W,W} \quad \text{and} \quad \Pi^{W,B} = \Pi^{B,W}. \]
		Moreover, $\Pi^{B,B}$ is a real skew-symmetric matrix, while $\Pi^{W,B}$ and $\Pi^{B,W}$ are purely imaginary symmetric matrices.
	\end{enumerate}
\end{corollary}

\begin{proof}
	(i) Follows directly from Corollary~\ref{cor:matrix_real_structure}(i) with $H=0$.
	
	(ii) Corollary~\ref{cor:matrix_real_structure}(ii) with $H=0$ implies $\Pi^{B,B} = -\Pi^{W,W}$ and $\Pi^{W,B} = \Pi^{B,W}$.
	Since $\tilde{\Pi}$ is symmetric (Proposition~\ref{prop:period_matrix}), its diagonal blocks $\Pi^{B,W}$ and $\Pi^{W,B}$ are symmetric. Its off-diagonal blocks satisfy $\Pi^{B,B} = (\Pi^{W,W})^T$. Substituting the relation from above yields $\Pi^{B,B} = (-\Pi^{B,B})^T$, proving skew-symmetry.
\end{proof}

\subsection{Construction of Examples}\label{sec:examples}

The linear nature of the discrete theory facilitates the implementation of algorithms to compute the (complete) discrete period matrix of a given discrete Riemann surface. Such computations were recently performed by \c{C}elik, Fairchild, and Mandelshtam in \cite{CFM23}. Provided the number of faces is not excessively large, the resulting system of linear equations can be solved efficiently by computer algebra systems.

We utilized Mathematica to perform a series of experiments, all of which corroborated our theoretical results in Section~\ref{sec:discrete_homology}. In particular, we experimentally confirmed that the analogue of Corollary~\ref{cor:matrix_real_structureM}~(i) does not hold for involutions that swap the bipartite coloring: There exist discrete M-curves with a color-reversing $\tau$ (Type 2) such that the complete discrete period matrix $\tilde{\Pi}$ is \emph{not} purely imaginary. However, the discrete period matrix $\Pi$ behaves as its continuous counterpart according to Corollary~\ref{cor:periodmatrix}.
We omit the detailed experimental data, which can be found in \cite{D23}, as it strictly aligns with the theoretical discussion above.

Instead, we present a constructive machinery to generate discrete real Riemann surfaces of any topological type (dividing or non-dividing) and with any number of real ovals allowed by Theorem~\ref{th:number_ovals_discrete}. On these surfaces, the computational methods of \cite{CFM23} can be applied directly. Convergence results in \cite{G23} establish convergence for finer discretizations of a fixed polyhedral surface. In light of this, we expect the resulting discrete period matrices to approximate their continuous counterparts.

Since the structure of the complete discrete period matrix in Theorem~\ref{th:completeperiodmatrix} exhibits a closer alignment with the continuous theory in the case of color-preserving involutions, we restrict our explicit constructions to this class (Type 1) in the following.

\subsubsection*{Delaunay-Voronoi quadrangulations}

Our method relies on constructing quadrangulated surfaces from specific triangulated surfaces, a technique studied by Bobenko and Skopenkov \cite{BoSk12}. Let $\Gamma$ be a triangulation of a Riemann surface $\Sigma$, realized as a polyhedral surface $S$ in Euclidean space.
If the intrinsic circumcircles of the triangles on $S$ contain no vertices other than the vertices of the respective triangle, the triangulation is called \textit{Delaunay}. In this case, the dual vertices $V(\Gamma^*)$ can be placed at the circumcenters of the triangles (which may lie outside the triangles themselves). Connecting the vertices of each triangle with the corresponding dual vertex yields a bipartite quad-graph $\Lambda$, known as the \textit{Delaunay-Voronoi quadrangulation} \cite{BoSk12}.
The discrete complex structure is inherited from the Euclidean geometry of $S$. Since the circumcenters are the intersection points of perpendicular bisectors, the resulting discrete Riemann surface is orthodiagonal ($\rho \in \mathds{R}^+$). Thus, the specialized results for the period matrix (Lemma~\ref{lem:period_orthodiagonal} and Corollary~\ref{cor:matrix_real_structure}) apply.

If $\Gamma$ is not Delaunay, we may simply place the dual vertices at arbitrary points in the interior of the triangles (e.g., barycenters). This construction still yields a bipartite quad-graph $\Lambda$ and a discrete Riemann surface $(\Sigma, \Lambda)$ induced by the Euclidean metric, though it will generally not be orthodiagonal.

We now describe constructions for the different topological types of discrete real Riemann surfaces.

\subsubsection*{Dividing discrete real Riemann surfaces}

Let $k$ and $g$ satisfy $k \equiv g+1 \pmod 2$, as required for the dividing case (Theorem~\ref{th:number_ovals_discrete}).
We start with a triangulated polyhedral surface $T$ in $\mathds{R}^3$ of genus $g' \coloneqq \frac{g+1-k}{2}$, which has $k$ boundary components lying on a common plane $\varepsilon$. We assume $T$ does not intersect $\varepsilon$ elsewhere. Reflecting $T$ across this plane yields a closed symmetric surface $S$ of genus $2g'+k-1 = g$ (as in Figure~\ref{fig:dividing6}, where $g'=2$ and $k=3$).

According to \cite{BoSp07}, the induced triangulation on $S$ is Delaunay if and only if the sum of angles opposite to any interior edge is at most $\pi$. Thus, if the initial triangulation $T$ satisfies this condition for interior edges, and if the angles opposite to boundary edges are non-obtuse, the resulting global triangulation is Delaunay.

We construct $\Lambda$ by placing dual vertices and reflecting them across $\varepsilon$. If the triangulation is Delaunay, we choose the dual vertices to be the circumcenters; in this case, the resulting discrete Riemann surface $(S, \Lambda)$ is orthodiagonal. If the triangulation is not Delaunay, we choose symmetric interior points, yielding a general discrete Riemann surface.
Since Euclidean reflections map circumcenters to circumcenters, the reflection across $\varepsilon$ induces a color-preserving involution $\tau$ on $\Lambda$. Geometric considerations show $\tau(\rho_Q) = \overline{\rho_Q}$, ensuring $\tau$ is discrete antiholomorphic. The fixed point set consists exactly of the $k$ boundary curves of $T$, making $(S, \Lambda, \tau)$ a dividing surface with $k$ ovals.

\subsubsection*{Non-dividing discrete real Riemann surfaces ($k > 0$)}

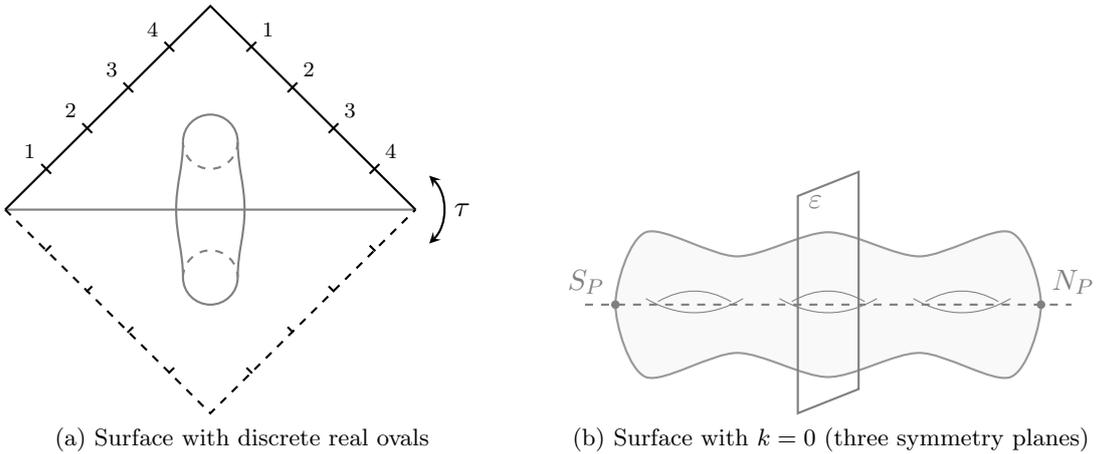
\begin{figure}[htbp]
	\centering
	\subfloat[Surface with discrete real ovals]{
		\begin{tikzpicture}[scale=0.9, >=Latex, thick]
			\coordinate (L) at (-3, 0);
			\coordinate (R) at (3, 0);
			\coordinate (T) at (0, 3);
			\coordinate (B) at (0, -3);
			
			\draw[gray] (L) -- (R);
			
			\draw[gray, thick] (-0.4, 1) to[out=-90, in=90] (-0.5, 0) to[out=-90, in=90] (-0.4, -1);
			\draw[gray, thick] (0.4, 1) to[out=-90, in=90] (0.5, 0) to[out=-90, in=90] (0.4, -1);
			
			\draw[gray, thick, fill=white] (-0.4, 1) arc (180:0:0.4);
			\draw[gray, thick, dashed] (0.4, 1) arc (0:-180:0.4);   
			
			\draw[gray, thick, fill=white] (0.4, -1) arc (0:-180:0.4); 
			\draw[gray, thick, dashed] (-0.4, -1) arc (180:0:0.4);  
			
			\draw (L) -- (T) -- (R);
			
			\foreach \p/\l in {0.2/1, 0.4/2, 0.6/3, 0.8/4} {

	\draw ($(L)!\p!(T)!0.1cm!90:(T)$) -- ($(L)!\p!(T)!-0.1cm!90:(T)$);
	\node[above left, font=\scriptsize] at ($(L)!\p!(T)$) {\l};
	\draw ($(T)!\p!(R)!0.1cm!90:(R)$) -- ($(T)!\p!(R)!-0.1cm!90:(R)$);
	\node[above right, font=\scriptsize] at ($(T)!\p!(R)$) {\l};
}
			
			\draw[dashed] (L) -- (B) -- (R);
			\foreach \p in {0.2, 0.4, 0.6, 0.8} {
				\draw[dashed] ($(L)!\p!(B)!0.1cm!90:(B)$) -- ($(L)!\p!(B)!-0.1cm!90:(B)$);
				\draw[dashed] ($(B)!\p!(R)!0.1cm!90:(R)$) -- ($(B)!\p!(R)!-0.1cm!90:(R)$);
			}
			
			\draw[<->, thick, >=stealth, bend left=45] (3.2, 0.5) to node[right] {$\tau$} (3.2, -0.5);
			
		\end{tikzpicture}
	}
	\qquad
	\subfloat[Surface with $k=0$ (three symmetry planes)]{
		\begin{tikzpicture}[scale=0.8, >=Latex, thick]
			\def\len{3.5} 
			\def\h{1.2}   
			
			\draw[gray!80, fill=gray!5] plot [smooth cycle] coordinates 
			{(-\len,0) (-\len+0.5, \h) (-1.5, 0.8) (0, \h) (1.5, 0.8) (\len-0.5, \h) 
				(\len,0) (\len-0.5, -\h) (1.5, -0.8) (0, -\h) (-1.5, -0.8) (-\len+0.5, -\h)};
			
			\foreach \x in {-2.2, 0, 2.2} {
				\begin{scope}[shift={(\x,0)}]
					\draw[thin, gray] (-0.8,0.1) to[out=-30, in=210] (0.8,0.1);
					\draw[thin, gray] (-0.6,0.05) to[out=30, in=150] (0.6,0.05);
				\end{scope}
			}
			
				\draw[gray, thick] (-0.5, -1.8) -- (-0.5, 1.8) node[right, yshift=-2pt] {$\varepsilon$} -- (0.5, 2.2) -- (0.5, -1.4) -- cycle;
			
			\draw[gray, thick, dashed] (-\len-0.5, 0) -- (\len+0.5, 0);
			
			\fill[gray] (-\len, 0) circle (2pt) node[above left] {$S_P$};
			\fill[gray] (\len, 0) circle (2pt) node[above right] {$N_P$};
			
		\end{tikzpicture}
	}
	\caption{Construction of non-dividing discrete real Riemann surfaces.}
	\label{fig:non-dividing_0}
\end{figure}

We construct a genus $g$ surface starting from a planar polygon. Let $g' \coloneq g - k + 1$. Consider a regular $4g'$-gon in the plane. We triangulate a fundamental domain corresponding to half of this polygon (cut by a main diagonal). For $g'=1$, this is an isosceles right triangle (half of a square).
We assume the boundary edges on one half of the perimeter match the triangulation on the opposite half to allow for identification. Reflecting the triangulation across the diagonal yields the full $4g'$-gon. Identifying opposite edges produces a flat polyhedral surface $S'$ of genus $g'$.

To introduce real ovals, we embed this flat surface into $\mathds{R}^3$ and position it symmetrically with respect to a plane $\varepsilon$ passing through the reflection diagonal. We then excise $k-1$ pairs of disjoint triangles symmetric with respect to $\varepsilon$ and replace them with triangulated handles connecting the holes to the plane $\varepsilon$. This operation adds $k-1$ to the genus, resulting in a total genus $g = g' + k - 1$.

We ensure that the Delaunay condition is met during the initial triangulation and the attachment of handles. If this condition is satisfied, we place the dual vertices at the circumcenters and reflect them accordingly. This construction yields an orthodiagonal discrete real Riemann surface. If the Delaunay condition is not met, generic dual vertices are used, resulting in a non-orthodiagonal surface.
The reflection across $\varepsilon$ induces the involution $\tau$. The fixed point set consists of the diagonal of the fundamental polygon and the $k-1$ intersection curves of the handles with $\varepsilon$, yielding exactly $k$ ovals. Since the fundamental polygon edges are identified in an orientation-reversing manner relative to the reflection, the surface is non-dividing.

\subsubsection*{Non-dividing discrete real Riemann surfaces without ovals ($k=0$)}

For the case $k=0$, we mimic the antipodal map $x \mapsto -x$. We consider a genus $g$ surface in $\mathds{R}^3$ that is \textit{highly symmetric}, meaning it is invariant under reflections across three orthogonal coordinate planes. For $g=0$ (the sphere), the composition of these three reflections is the antipodal map, which has no fixed points.

Figure~\ref{fig:non-dividing_0}b depicts such a symmetric surface of genus $g$. For visual clarity, only one symmetry plane $\varepsilon$ (the equatorial plane) and the rotation axis passing through the north and south poles ($N_P, S_P$) are shown. The involution $\tau$ can be realized as the reflection across $\varepsilon$ followed by a $180^\circ$ rotation around this axis. This operation is equivalent to the composition of reflections across the other two coordinate planes. Since the rotation has no fixed points in the equatorial plane, and the reflection fixes only the equatorial plane, the composition has no fixed points on the surface.

For the discrete construction, we generate a symmetric polyhedral surface by starting with a triangulation $T$ of a fundamental domain in the first octant (one eighth of the space), such that the boundary lies entirely on the three coordinate planes. Reflecting $T$ successively across the $xy$-, $yz$-, and $zx$-planes generates a closed surface $S$ of genus $g$.

Provided the initial triangulation $T$ satisfies the angle conditions (non-obtuse angles opposite to boundary edges), the resulting global triangulation is Delaunay. In this case, placing dual vertices at the circumcenters produces an orthodiagonal discrete Riemann surface.
The involution $\tau$ is defined as the composition of the three reflections (central symmetry). As in the continuous case, this $\tau$ is discrete antiholomorphic (color-preserving) but has no fixed points, as the origin is not part of the surface. Thus, we obtain a non-dividing discrete real Riemann surface with $k=0$.


\appendix
\section{Symmetric matrices over $\mathds{Z}_2$}\label{sec:bilinearforms}

Throughout this section, we consider matrices over the field $\mathds{Z}_2$.
Our aim is to give an elementary proof of the known fact that symmetric matrices over $\mathds{Z}_2$ are uniquely characterized by their rank and the presence of non-zero diagonal elements (denoted by $\mathrm{diag}$) up to congruence transformations.
Any symmetric matrix defines a symmetric bilinear form; we follow the notations and ideas given in \cite{Alb38}, adapted to our specific setting.

Let $A, B$ be two $g \times g$-matrices over $\mathds{Z}_2$.
We say that $A$ and $B$ are \emph{congruent} if there exists an invertible matrix $P$ such that
\[
A = P B P^{T}.
\]
Clearly, congruence preserves symmetry.
To establish the normal forms, we utilize elementary row and column operations: replacing the $i$-th row by the sum of itself and a linear combination of other rows must be accompanied by the corresponding operation on the columns to maintain symmetry (this corresponds to $P$ performing the row operations). This process is the \emph{symmetric Gauss algorithm}.

We call a submatrix $G$ of $A$ a \emph{principal submatrix} if the main diagonal of $G$ is a subset of the main diagonal of $A$.

\begin{lemma}\label{lem:congruence}
	Every non-zero symmetric matrix $A$ is congruent to a matrix that has a non-zero principal $1\times1$- or $2\times2$-submatrix.
\end{lemma}

\begin{proof}
	If there exists an index $i$ such that $a_{ii} \neq 0$, then $A$ already has a non-zero principal $1\times1$-submatrix.
	
	Assume now that $a_{ii} = 0$ for all $i$. Since $A \neq 0$, there exist indices $i < j$ such that $a_{ij} = a_{ji} \neq 0$.
	If $j = i + 1$, the principal submatrix defined by rows and columns $i, i+1$ is $\begin{pmatrix} 0 & 1 \\ 1 & 0 \end{pmatrix}$, which is non-zero.
	
	If $j > i + 1$, we apply the congruence transformation that adds the $i$-th row to the $(j-1)$-th row (and simultaneously the $i$-th column to the $(j-1)$-th column).
	The new entry at position $(j-1, j)$ becomes
	\[
	a'_{j-1, j} = a_{j-1, j} + a_{i, j} = 0 + 1 = 1.
	\]
	Since the new matrix has a non-zero entry at $(j-1, j)$, it possesses a non-zero principal $2\times2$-submatrix at indices $j-1, j$.
\end{proof}

\begin{lemma}\label{lem:invertible_sub}
	Let $G$ be an invertible principal submatrix of a symmetric matrix $A$.
	Then $A$ is congruent to a block diagonal matrix
	\[
	\begin{pmatrix}
		G & 0 \\[0.2em]
		0 & H
	\end{pmatrix}.
	\]
\end{lemma}

\begin{proof}
	By applying a permutation, we can move the submatrix $G$ to the upper-left corner.
	Since $G$ is invertible, we can use it as a pivot block: adding suitable multiples of the rows of $G$ to the rows below (and symmetrically for columns) eliminates all entries in the lower-left and upper-right blocks, leaving a matrix $H$ in the lower-right.
\end{proof}

\begin{proposition}\label{prop:full_principal}
	Every symmetric matrix is congruent to a block diagonal matrix of the form
	\[
	\begin{pmatrix}
		I_k & 0 & 0 \\[0.2em]
		0 & G_m &0\\[0.2em]
		0 & 0 & 0
	\end{pmatrix},
	\]
	where $I_k$ is the $k \times k$ identity matrix, and $G_m$ consists of $m$ diagonal blocks of the form $\begin{pmatrix} 0 & 1 \\ 1 & 0 \end{pmatrix}$.
\end{proposition}

\begin{proof}
	We proceed by induction on the size of the matrix.
	If $\mathrm{diag}(A) \neq 0$, there is a non-zero diagonal entry (a $1 \times 1$ principal submatrix $G=(1)$). Since $(1)$ is invertible, Lemma~\ref{lem:invertible_sub} allows us to split it off and recurse on the remainder. This produces identity blocks.
	
	If $\mathrm{diag}(A) = 0$ but $A \neq 0$, Lemma~\ref{lem:congruence} ensures the existence of a principal $2 \times 2$ submatrix $G = \begin{pmatrix} 0 & 1 \\ 1 & 0 \end{pmatrix}$. Since $\det(G) = 1 \neq 0$, it is invertible. Again, Lemma~\ref{lem:invertible_sub} allows us to split $G$ off and recurse. This produces blocks of type $\begin{pmatrix} 0 & 1 \\ 1 & 0 \end{pmatrix}$.
	
	Iterating this process until the remainder is zero yields the stated block form.
\end{proof}

\begin{corollary}\label{cor:characterization}
	Every symmetric matrix over $\mathds{Z}_2$ is uniquely characterized by its rank and its diagonal (whether it vanishes or not) up to congruence.
\end{corollary}

\begin{proof}
	Let $A$ be a symmetric matrix.
	\begin{enumerate}
		\item If $\mathrm{diag}(A) = 0$, then in the normal form of Proposition~\ref{prop:full_principal}, no identity blocks $I_k$ can appear (as they would introduce non-zero diagonal elements). Thus, $A$ is congruent to a direct sum of $m$ blocks of $\begin{pmatrix} 0 & 1 \\ 1 & 0 \end{pmatrix}$. The rank is $2m$. Since the rank is invariant under congruence, $A$ is uniquely determined by $\mathrm{rank}(A)$. (Note that for $\mathrm{diag}(A)=0$, the rank must be even).
		
		\item If $\mathrm{diag}(A) \neq 0$, the normal form contains at least one identity block ($k \ge 1$). We observe that the direct sum of an identity block and a symplectic block is congruent to three identity blocks:
		\[
		(1) \oplus \begin{pmatrix} 0 & 1 \\ 1 & 0 \end{pmatrix}
		=
		\begin{pmatrix}
			1 & 0 & 0 \\
			0 & 0 & 1 \\
			0 & 1 & 0
		\end{pmatrix}
		\cong
		\begin{pmatrix}
			1 & 0 & 0 \\
			0 & 1 & 0 \\
			0 & 0 & 1
		\end{pmatrix} = I_3.
		\]
		Explicitly, the transformation is given by $B = P A P^T$ with
		\[
		P = \begin{pmatrix}
			1 & 1 & 0 \\
			1 & 0 & 1 \\
			1 & 1 & 1
		\end{pmatrix}, \quad
		\begin{pmatrix}
			1 & 1 & 0 \\
			1 & 0 & 1 \\
			1 & 1 & 1
		\end{pmatrix}
		\begin{pmatrix}
			1 & 0 & 0 \\
			0 & 0 & 1 \\
			0 & 1 & 0
		\end{pmatrix}
		\begin{pmatrix}
			1 & 1 & 1 \\
			1 & 0 & 1 \\
			0 & 1 & 1
		\end{pmatrix}
		=
		\begin{pmatrix}
			1 & 0 & 0 \\
			0 & 1 & 0 \\
			0 & 0 & 1
		\end{pmatrix}.
		\]
		Using this relation recursively, we can convert all blocks of type $\begin{pmatrix} 0 & 1 \\ 1 & 0 \end{pmatrix}$ into identity blocks $I_2$, provided there is at least one $I_1$ to start with.
		Thus, if $\mathrm{diag}(A) \neq 0$, $A$ is congruent to $I_r \oplus 0$, where $r = \mathrm{rank}(A)$.
	\end{enumerate}
	Consequently, the rank and $\mathrm{diag}(A)$ completely determine the congruence class.
\end{proof}

\bibliographystyle{plain}
\begin{small}
	\bibliography{Literature}	
\end{small}
\end{document}